\documentclass[10pt]{article}
\usepackage{pgf,tikz,pgfplots}
\pgfplotsset{compat=1.15}
\usepackage{mathrsfs}
\usepackage[utf8]{inputenc}
\usepackage[english]{babel}
\usepackage{hyperref}
\usepackage{amsthm, amsmath}
\usepackage{enumitem}

\usetikzlibrary{arrows}

\usepackage{fullpage}
\usepackage{graphicx}
\usepackage{subfigure}

\newtheorem{theorem}{Theorem}[section]

\newtheorem{theorem2}[theorem]{Theorem}
\newtheorem{definition}[theorem]{Definition}
\newtheorem{proposition}[theorem]{Proposition}
\newtheorem{corollary}[theorem]{Corollary}
\newtheorem{remark}[theorem]{Remark}
\newtheorem{lemma}[theorem]{Lemma}

\usepackage{amsmath}
\usepackage{amssymb}
        \def\X{\mathrm{X}}
        
        \def\S{\mathrm{S}}
\usepackage{url}
\usepackage{graphicx}

\newcommand\N{\mathbb{N}}
\newcommand\R{\mathbb{R}}

        \newcommand{\HH}{{\mathcal H}}
        \renewcommand{\H}{\HH^1}

   \newcommand{\defeq}{:=}
        \def\dist{\mathrm{dist}\,}
        \def\conv{\mathrm{conv}\,}
        \def\oord{\mathrm{ord}\,}
        \def\oordball{\mathrm{ordball}\,}
        \def\Int{\mathrm{Int}\,}
        \def\Relint{\mathrm{Relint}\,}
        \def\diam{\mathrm{diam}\,}
       \def\turn{\mathrm {turn}\,}

\title{On regularity of maximal distance minimizers \\}

\author{Yana Teplitskaya\footnote{Chebyshev Laboratory, St. Petersburg State University, 14th Line V.O., 29, Saint Petersburg 199178 Russia.} }

\begin{document}

\maketitle

\renewcommand{\abstractname}{Abstract}
\begin{abstract}
We study the properties of sets $\Sigma$ which are the solutions of the maximal distance minimizer problem, id est of sets
having the minimal length (one-dimensional Hausdorff measure) over the class	of closed connected sets $\Sigma \subset \mathbb{R}^2$ satisfying the inequality $\max_{y \in M} \dist(y,\Sigma) \leq r$ for a given compact
set $M \subset \mathbb{R}^2$ and some given $r > 0$. 
Such sets can be considered as the shortest networks of radiating cables arriving to each customer (from the set $M$ of customers) at a distance at most $r$.
	
In this work it is proved that each maximal distance minimizer is a union of finite number of simple curves, having one-sided tangents at each point. Moreover the angle between these rays at each point of a maximal distance minimizer is greater or equal to $2\pi/3$. 
It shows that a maximal distance minimizer is isotopic to a finite Steiner tree even for a ``bad'' compact $M$, which differs it from a solution of the
Steiner problem (there exists an example of a Steiner tree with an infinite number of branching points).
Also we classify the behavior of a minimizer in a neighbourhood of an arbitrary point of $\Sigma$.


In fact, all the results are proved for more general class of local minimizer, i.~e. sets which are optimal in a neighbourhood of its arbitrary point.
\end{abstract}

\section{Introduction}

For a given compact set $M \subset \mathbb{R}^2$ consider the functional
\[
	F_{M}(\Sigma)\defeq \sup _{y\in M}\dist (y, \Sigma),
\]
where $\Sigma$ is a closed subset of $\R^2$ and $\dist(y, \Sigma)$ stands for the Euclidean distance between $y$ and $\Sigma$ (naturally, $F_{M} (\emptyset) \defeq +\infty$). The quantity $F_M (\Sigma)$ will be called the \emph{energy} of $\Sigma$.
Consider the class of closed connected sets $\Sigma \subset \R^2$ satisfying $F_M(\Sigma) \leq r$ for some $r > 0$. We are interested in the properties of
sets of the minimal length (one-dimensional Hausdorff measure) $\H(\Sigma)$ over the mentioned class. Such sets will be further called \emph{minimizers}. 

It is known that for all $r>0$ the set  of minimizers is nonempty. It is proven also that for each minimizer of positive length the equality $F_M(\Sigma)=r$ holds. Furthermore the set of minimizers coincides with the set of solutions of the dual problem: minimize $F_M$ over all closed connected sets  $\Sigma \subset \mathbb{R}^2$ with prescribed bound on the total length $\H(\Sigma) \leq l$. Thus the required sets are called maximal distance minimizers.
For the rigorous statements and proofs of these claims, including the general $n$-dimensional case, see~\cite{mir}.

Let $B_\rho (x)$ be the open ball of radius $\rho$ centered at a point $x$, and let $B_\rho(M)$ be the open $\rho$-neighborhood of $M$ i.e.\
	\[
        		B_\rho(M) \defeq \bigcup_{x\in M} B_\rho(x).
	\]
	
It is easy to see that the set $\Sigma$ is bounded (and, thus, compact), since $\Sigma \subset \overline{B_r(\conv M)}$, where $\conv M$ stands for the convex hull of the set $M$. 
 
 \begin{definition}
A point $x \in \Sigma$ is called \emph{energetic}, if for all $\rho>0$ one has
 	\[
		F_{M}(\Sigma \setminus B_{\rho}(x)) > F_{M}(\Sigma).
	\]
\end{definition}
Denote the set of all energetic points of $\Sigma$ by $G_\Sigma$.

Each minimizer $\Sigma$ can be split into three disjoint subsets:
        \[
        \Sigma=E_{\Sigma}\sqcup\X_{\Sigma}\sqcup\S_{\Sigma},
        \]
        where $X_{\Sigma}\subset G_\Sigma$ 
is the set of isolated energetic points 
(i.e.\ every $x\in X_\Sigma$ is energetic and there is a $\rho>0$ possibly depending on $x$ such that $B_{\rho}(x)\cap G_{\Sigma}=\{x\}$), $E_{\Sigma} := G_{\Sigma}\setminus X_{\Sigma}$ is the set of non isolated energetic points and $S_\Sigma \defeq \Sigma \setminus G_\Sigma$ is the set of non energetic points also called the \textit{Steiner part}.
Clearly, the limit points of $G_\Sigma$ belong to $E_\Sigma$.

 In~\cite{mir} (for the plane) and in~\cite{PaoSte04max} (for the general case) the following properties of minimizers have been proven:
        \begin{itemize}
        \item[(a)] A minimizer cannot contain loops (homeomorphic images of circles).
        \label{aaaa}
        \item[(b)] For every point $x \in G_\Sigma$ there exists a point $y \in M$ (may be not unique) such that $\dist (x, y) = r$ and $B_{r}(y)\cap \Sigma=\emptyset$.
         \label{bbbb}
        \item[(c)] For all $x \in S_\Sigma$ there exists an $\varepsilon>0$ such that $S_\Sigma \cap B_{\varepsilon}(x)$ is either a segment or a regular tripod, i.e.\ the union of three segments with an endpoint in $x$ and relative angles of $2\pi/3$. 
       
        \end{itemize}

Also, since $\Sigma$ is connected closed set of finite length, the following statements hold. 
 	\begin{itemize}
	\item[(d)]
	$\Sigma$ has a tangent line at almost every point of $\Sigma$ 
    \label{e}
    \item[(f)] Set $\Sigma$ is Ahlfors regular, id est there exists such constants $c>0$ and $C>0$ that for every point $x \in \Sigma$ or sufficiently small $\rho>0$  the inequality
    \[
    c\rho \leq \H(\Sigma \cap B_\rho(x))\leq C\rho.
    \]
    \label{j}
   In this work we find exact constants $c$ and $C$ for each point $x$: we will show that \[
   \H(\Sigma \cap B_\rho(x))=\oord_x\Sigma \rho+o(\rho)
   \]
   where $\oord_x\Sigma$ is equal to $1$, $2$ or $3$.
\end{itemize}

The main result of this article is formulated in Theorem~\ref{mainT1}: let $\Sigma$ be a maximal distance minimizer for a compact set $M \subset \mathbb{R}^2$. Then \begin{itemize}
\item[(i)] $\Sigma$ is a union of a finite number of injective images of the segment.
\item[(ii)] The angle between each pair of tangent rays at every point of $\Sigma$ is greater or equal to $2\pi/3$.
\item[(iii)] By (ii) the number of tangent rays at every point of $\Sigma$ is not greater than $3$. If it is equal to $3$, then there exists such a neighbourhood of $x$ that the arcs $\Sigma$ coincide with three line segments and by (ii) the pairwise angles between them are equal to $2\pi/3$. 
\end{itemize}

Theorem~\ref{mainT1} implies that the number of branching points is always finite for each maximal distance minimizer. Note that this is not true for the Steiner problem: an example of Steiner tree with infinite number of branching points can be found in~\cite{paolini2015example}.

We should also mention the problem of minimizing the average distance functional (for more details about this problem, see the survey~\cite {L} and~\cite{lemenant2011regularity}). Maximum and average distance minimizers have a number of similar geometric and topological properties (see~\cite{PaoSte04max}). The finiteness of the number of branching points for the minimizers of the average distance was proved in~\cite{BS1}.

The article actually proved a number of geometric properties of minimizers, but in order not to clutter up the narrative, we will not dwell on them in this work. The article is organized as follows. Section~\ref{stp} contains the statement of the Steiner problem and some technical statements about Steiner trees (the advanced reader can skip this chapter); Chapter~\ref{not} introduces used designations, basic definitions and elementary statements; Chapter~\ref{res} contains the statement and proof of the main result.



\subsection{Steiner problem} \label{stp}
Let the finite set of points $C \defeq \{ A_1, \dots, A_n \} \subset \R^2$ is given. Then the shortest (having a minimal one-dimensional Hausdorff measure) connected (contained) it set $S\subset \R$ is called a \textit{Steiner tree}. It is known that Steiner tree exists (but may be not unique) and coincides with a finite union of straight segments. So we can identify $S$ with a plane graph whose vertex set contains $ C $ and whose edges are rectilinear. This graph is connected and does not contain cycles, so it is a tree. It is known that the vertices of this graph cannot have degree greater than $ 3 $. Wherein only the vertices $ A_i $ can have degrees $ 1 $ and $ 2 $, the remaining vertices have degrees $ 3 $ and are called \textit{Steiner points} (or \textit{branching points}). Then the number of Steiner vertices does not exceed $ n-2 $. The pairwise angles between edges incident to an arbitrary vertex are at least $ 2 \pi / 3 $.

Thus, a Steiner point has exactly three edges with pairwise angles $ 2 \pi / 3 $. If all vertices $ A_i $ have degree $ 1 $, then the number of Steiner points is equal to $ n-2 $ (the converse is also true) and $ S $ is called a \textit{full} Steiner tree. A set, each connected component of which is a (full) Steiner tree, is called a \textit{(full) Steiner forest}.
For a given three-point set $C=\{A_1, A_2, A_3\}$ Steiner tree is unique and is a
\begin{itemize}
    \item union of two segments $[A_kA_l] \cup [A_lA_m]$, if $\angle A_k A_l A_m \geq 2\pi/3$, where $\angle A_kA_lA_m$ is a maximal angle of the triangle $\triangle A_1A_2A_3$, for $\{k,l,m\}=\{ 1,2,3\}$.  
    \item union of three segments $[A_1F] \cup [A_2F] \cup [A_3F]$ otherwise, where  $F$ is (unique) such point that \[\angle A_1FA_2= \angle A_2FA_3=\angle A_3FA_1=2\pi/3.\]
\end{itemize}

The proof of the listed properties of Steiner trees and other interesting information about them one can find in the book~\cite{hwang1992steiner} and in~\cite{gilbert1968steiner}. The problem can also be posed in a more general form: let $ C $ be a compact subset of the metric space, and required set $ S $ has the minimum one-dimensional Hausdorff measure among all sets $ S'$ such that $ S' \cup C $ is connected. In this case, the set $ \bar{S} $ will be called a solution to the Steiner problem and will have some properties. In particular, it is known that $ \bar S $ exists, $ \bar{S} $ does not contain cycles, and $ S \setminus C $ is at most countable union of geodesic curves of a metric space. See~\cite{PaoSte13Steiner} for more details.

 
Finally, by a \textit{local Steiner forest (tree)} we mean a (connected) compact set $ S $ such that for an arbitrary point $ x \in \Relint (S \setminus C) $ (where $ \Relint $ stands for the relative interior of the set) there exists an open neighborhood $ U \ni x $ such that $ S \cap \overline U $ coincides with the Steiner tree connecting the points $ S \cap \partial U $. A local Steiner tree has the following properties of a Steiner tree: it consists of a countable set of geodesic segments, i.e. in the case of the Euclidean plane, it consists of segments, and the angles between the segments incident to the same point of $ S $ are at least $ 2 \pi / 3 $.

\begin{proposition}
Let $ A $, $ B $ and $ C $ be three different points of the plane such that all angles of the triangle $ \triangle ABC $ are strictly less than $ 2 \pi / 3 $, and the lengths of the sides $ AB $ and $ BC $ are equal $ \varepsilon $. Then 
\[
2 \varepsilon - \H (St (A, B, C)) = d \varepsilon,
\]
where $ St(A, B, C) $ is a Steiner tree for the set $ \{A, B, C \} $ and $ d > 0 $ is a constant depending only on the angles of the triangle $ \triangle ABC $.
\label{angles4}
\end{proposition}

\begin{definition}
We define the ``wind rose'' as six rays from the origin, with an angle $ \pi / 3 $ between any two adjacent ones, equipped with weights (real numbers) that have the following property: the weight on each ray is the sum of the weights of the rays adjacent to it. (From this, in particular, it follows that the weights on opposite (forming a straight line) rays add up to zero.)
\label{roza_v}
\end{definition}
An example is shown in Figure~\ref{fig:roza1}. A similar construction, but for physical reasons, appears in~\cite{gilbert1968steiner}.

\begin{remark}
A full Steiner tree consists of segments belonging to three straight lines with pairwise angles $ \pi / 3 $. Let $ T $ be an arbitrary full Steiner tree. Take an arbitrary wind rose, consisting of rays, parallel to the segments of $ T $, and assign to each end vertex of $ T $ the weight of the ray of this wind rose, co-directed with the segment entering this vertex. Then the sum of numbers over all vertices of $ T $ will be zero.
\label{main_rose_property}
\end{remark}

\begin{remark}
Remark~\ref{main_rose_property} is also true for a local Steiner tree as well as for a full Steiner forest and a local Steiner forest (in the case of forests, multiple wind roses may be required to assign weights to).
\end{remark}

\begin{lemma}
Let $ S $ be a full Steiner forest and  let $ l $ be an arbitrary line such that no connected component of $ S $ is contained in $ l $. Then
\[
\sharp(\partial S \cap l) \leq 2\sharp(\partial S \setminus l),
\]
where $\partial S$ is a vertex set of $S$ 
\label{St_l_c}
\end{lemma}

\begin{proof}
Let $ S $ have a finite number of vertices. Let $ l^{+} $, $ l^{-} $ be open half-planes bounded by $ l $. It is easy to see that it is sufficient to separately prove the inequalities for $ \overline {S \cap l^{+}} $, $ \overline {S \cap l^{-}} $. Moreover, it suffices to prove the assertion for one connected component of these sets. Thus, we will prove the assertion for the closure of an arbitrary Steiner tree lying in an open half-plane bounded by $ l $.
We place the center of the wind rose in the same open half-plane, choose a wind rose such that all those and only those rays that intersect $ l $ have a positive weight: obviously, such a wind rose exists, since $ l $ intersect $ 2 $ or $ 3 $ in a row going ray. In the first case, we will equip them with weights $ 1, 2, 1 $, in the second $ 1, 1 $. Then the rest of the rays will have weights $ -1, -2, -1 $ or $ 0, -1, -1, 0 $. Let's assign weights to the leaf vertices in the same way as it was done in Remark~\ref{main_rose_property}. Then the sum over the endpoints of $ l $ is at least $ \sharp (S \cap l) $. Due to Remark~\ref{main_rose_property} the sum over all endpoints must be zero, so endpoints not belonging to $ l $ must be at least $ \frac{\sharp (S \cap l)}{2} $.
Let $ S $ have an infinite number of vertices (but still, by the general statement of Steiner problem, is contained in a compact and has a finite length). Assume the statement opposite to the lemma, that is, that $ \sharp (S \cap l) = \infty $, and $ \sharp (\partial S \setminus l) $, on the contrary, is finite. It suffices to prove the statement for each of the half-planes $ l^{+} $ and $ l^{-} $. Without loss of generality, there is a connected component $ S \cap l^{+} $ such that its closure $ T $ has an infinite number of vertices on $ l $. Let $ l'$ be parallel to $ l $ and contained in $ l^{+} $. We will move the straight line $ l'$ inside this half-plane, parallel to $ l $, and denote by $ S (l') $ the Steiner forest, which is the intersection of $ T $ and the half-plane bounded by $ l'$ and not containing $ l $. 
For an arbitrarily large $ N $, there is $ l_0 $ such that $ N <\sharp (l_0 \cap T) <\infty $. Let's choose this $ l_0 $ for $ N = 3 \sharp (\partial T \setminus l) $. This contradicts the already proven statement for $ S (l_0) $.

\end{proof}

\begin{corollary}
Let $ T $ be a full (that is, having no vertices of degree $ 2 $) Steiner tree, $ l $ an arbitrary line such that $ T \not \subset l $. Then
\[  
\sharp(\partial T \cap l)\leq 2\sharp(\partial T \setminus l).
\]
\label{St_l}
\end{corollary}
\section{Notations and definitions}\label {not}
For a given set $X$ the sets $\overline{X}$, $\Int (X)$ and $\partial X$ represent accordingly its closure, interior and boundary.

Let $B$ and $C$ be points. Then $[BC]$, $]BC]$ and $]BC[$ denote closed, semi-closed and open interval accordingly, $[BC)$ and $(BC)$ represent the ray and the line containing these points.

The symbol $\angle$, unless otherwise  stated, denotes angle belonging to $[0,\pi]$  for ray and points (or to $[0,-\pi/2]$ for lines). The symbol $\widehat{}$  unless otherwise  stated, denotes order sensitive angle belonging to $(-\pi;\pi]$ for ray and points (or to $[-\pi/2,-\pi/2]$ for lines), where the counterclockwise direction is positive. Both symbols are applicable to two lines, two rays or three points.
For a (usually finite) set $ M $, the number of elements of the set will be denoted by $ \sharp M $.
\begin{definition}
We say that the order of a point $ x $ in the set $ \Sigma $ is $ n $ (and write $\oord_x \Sigma = n $) if there exists such $ \varepsilon_0 = \varepsilon_0 (x)> 0  $, that for of any open set $ U $ such that $ x \in U $ and $ \diam U < \varepsilon_0 $ the cardinality of the intersection $ \sharp (\partial U \cap \Sigma) $ is at least $ n $, and for any $ \varepsilon <\varepsilon_0 $ there is an open set $ V \ni x $ with $ \diam V <\varepsilon $ and $ \sharp (\partial V \cap \Sigma) = n $.
\end{definition}

Note that, since $ \Sigma $ is connected, for any point $ x \in \Sigma $, the order $ \oord_x \Sigma> 0 $ (if $ \Sigma $ contains more than one point).

\begin{definition}
We say that $ \oordball_x \Sigma = n $ (the circular order $\oordball_x \Sigma $ of the point $ x $ in the set $ \Sigma $ is $ n $) if there exists $ \varepsilon_0> 0 $ such that for any $ \delta <\varepsilon_0 $ the inequality $ \sharp (\partial B_\delta (x) \cap \Sigma) \geq n $ holds, and for any $ \varepsilon <\varepsilon_0 $ there is such $ \eta <\varepsilon $ that $ \sharp (\partial B_{\eta} (x) \cap \Sigma) = n $.
\end{definition}

\begin{remark}
The inequality $\oordball_x \Sigma \geq \oord_x \Sigma$ holds.
\label{rem_ord}
\end{remark}

\begin{proof}
Suppose the opposite:
\[
k \defeq \oordball_x \Sigma  < n \defeq \oord_x \Sigma. 
\]
Then in view of the definition of $\oord$,
\[
\exists \varepsilon_0: \forall U\ni x: \diam U <\varepsilon_0 \quad \sharp (\partial U \cap \Sigma) \geq n
\]
Thus for each $\delta<\frac {\varepsilon_0} 2$ holds
\[
 \sharp (\partial B_{\delta}(x) \cap \Sigma) \geq n.
\]
But on the other hand, in view of the definition of $\oordball$, there exists such $\delta_0<\varepsilon_0/2$, that 
\[
 \sharp (\partial B_{\delta_0}(x) \cap \Sigma) =k<n.
\]
We got a contradiction.
\end{proof}

\begin{remark}
For every $x\in \Sigma$ the inequality $\oord_x\Sigma \leq \oordball_x\Sigma\leq 2\pi<7$ holds due to coarea inequality. Otherwise the set $\Sigma \setminus B_\varepsilon(x) \cup \partial B_\varepsilon(x)$ for sufficiently small $\varepsilon >0$ is better than $\Sigma$: it has not greater energy and strictly less length than $\Sigma$, and is still connected, so we have a contradiction with optimality of $\Sigma$.
\end{remark}

\begin{definition}
The \textit{arc} $ \breve {[ax]} \subset \Sigma $ ($\breve {]ax]}$ or $\breve {]ax[}$) 
 is the continuous injective image of the closed (semi-closed or open, respectively) interval. If inclusion and exclusion of the ends of the arc don't matter we will write $\breve{ax}$.

\end{definition}

\begin{definition}
We say that the ray $ (ax] $ is the \textit{tangent ray} of the set $ \Sigma $ with a vertex at the point $ x\in \Sigma $ (or shortly \textit{of $\Sigma$ at $x$}) if there exists non  stabilized sequence of points $ x_k \in \Sigma $ such that $ x_k \rightarrow x $ and $ \angle x_kxa \rightarrow 0 $.
Sometimes, instead of converging angles, it will be more convenient for us to write convergence of rays: $ (x_k x] \rightarrow (a x] $.
\end{definition}

Let's repeat the key definitions already introduced in the previous section.
 \begin{definition}
A point $x \in \Sigma$ is called \emph{energetic}, if for all $\rho>0$ one has
 	\[
		F_{M}(\Sigma \setminus B_{\rho}(x)) > F_{M}(\Sigma)
	\]
	where 
	\[
		F_{M}(\Sigma)\defeq \sup _{y\in M}\dist (y, \Sigma).
	\]
\end{definition}
The set of all energetic points of $\Sigma$ is denoted by $G_\Sigma$.

Each minimizer $\Sigma$ can be split into three disjoint subsets:
        \[
        \Sigma=E_{\Sigma}\sqcup\X_{\Sigma}\sqcup\S_{\Sigma},
        \]
        where $X_{\Sigma}\subset G_\Sigma$ 
is the set of isolated energetic points 
(i.e. every $x\in X_\Sigma$ is energetic and there is a $\rho>0$ possibly depending on $x$ such that $B_{\rho}(x)\cap G_{\Sigma}=\{x\}$), $E_{\Sigma} := G_{\Sigma}\setminus X_{\Sigma}$ is the set of non isolated energetic points and $S_\Sigma \defeq \Sigma \setminus G_\Sigma$ is the set of non energetic points also called the \textit{Steiner part}.

 For every point $x \in G_\Sigma$ there exists a point $y \in M$ (may be not unique) such that $\dist (x, y) = r$ and $B_{r}(y)\cap \Sigma=\emptyset$. 
 We will refer to this property as \textit{basic property of energetic points}.
 
\begin{remark}
Due to the definition of energetic points for each non energetic point $x$ there exists such $\rho>0$ that $\overline{B_\rho(x) \cap \Sigma}$ is a Steiner forest 
with $C \subset \partial B_\rho(x)$. Thus there exists such $\varepsilon>0$ that $\overline{B_\varepsilon(x) \cap \Sigma}$ is a segment or a regular tripod centered at $x$, in the case of a tripod we will call $x$ a  \textit{branching} (or Steiner) point. So  for each non energetic point $x$ there exists such $\rho_1>0$ that $\overline{B_\rho(x) \cap \Sigma}$ is a Steiner tree.
\label{stpt}
\end{remark}

\begin{definition}
We say that a ray $ (ax] $ is the \textit{energetic ray} of the set $ \Sigma $ with a vertex at the point $ x\in \Sigma $ if there exists non  stabilized sequence of energetic points $ x_k \in G_\Sigma $ such that $ x_k \rightarrow x $ and $ \angle x_kxa \rightarrow 0 $.

\end{definition}

\begin{remark}
A point $x \in \Sigma$ belongs to the set of non isolated energetic points $E_\Sigma$ iff there exists an energetic ray with vertex at $x$.
\end{remark}

\begin{remark}
Each energetic ray is also a tangent ray.
\end{remark}

\begin{remark}
Let $\{x_k\} \subset G_\Sigma$ and let $x\in E_\Sigma$ be the limit point of $\{x_k\}$:  $x_k \rightarrow x$.
By basic property of energetic points for every point $x_k \in G_\Sigma$ there exists a point $y_k \in M$ (may be not unique) such that $\dist (x_k, y_k) = r$ and $B_{r}(y_k)\cap \Sigma=\emptyset$. In this case we will say that $y_k$ \textit{corresponds} to $x_k$.

Let $y$ be an arbitrary limit point of the set $\{y_k\}$. Then the following statements are true:

\begin{enumerate}
\item Set $\Sigma$ doesn't intersect $r$-neighbourhood of $y$: $B_r(y) \cap \Sigma =\emptyset$. The point $y$ belongs to $M$ and corresponds to $x$.   
\item Let $(zx)$ be orthogonal to $(xy)$. Then it contains the energetic ray of $\Sigma$ with a vertex at $x$. 

Without loss of generality let $(zx]$ be energetic ray.
\item Let us define an \textit{orientation} (clockwise or counterclockwise) of $(zx]$ such way that the orientated angle in this direction from $(zx]$ to $(yx]$ is equal to $\pi/2$.
 \label{orientation}

It is easy to notice that then the orientated angle in this direction from $(zx]$ to another energetic ray (if it exists) with a vertex at $x$, is not less than $\pi$.

\item Let $x \in E_\Sigma$. Then there exists one or two energetic ray with vertex at $x$.
\label{netri}
\item
One energetic ray $(zx]$ can have two orientations (we say that it is orientated in the both directions or two-orientated) if there exist two different limit points of $y_k$, where $y_k$ corresponds to $x_k$. Otherwise it is one-orientated.

\item In the case of two-orientated $(zx]$ each energetic ray is perpendicular to the line, connecting the limit points of $y_k$. 
The angle between $(zx]$ and another energetic ray (if it exists) is equal to $\pi$. 

\end{enumerate}
\label{emptyball}
\end{remark}

\begin{definition}
We say that the line $ (ax) $ is the \textit{tangent line} of the set $ \Sigma $  at the point $ x\in \Sigma $ if for each non  stabilized sequence of points $ x_k \in \Sigma $ such that $ x_k \rightarrow x $ holds $\angle ((x_kx),(ax)) \rightarrow 0$. 
\end{definition}

\begin{definition}
We say that the ray $ (ax] $ is a \textit{one-sided tangent} of the set $ \Sigma $ at the point $ x\in \Sigma $ if there exists such a connected component of $\Sigma \setminus \{x\}$, let us call it $\Sigma'$, that  for each non  stabilized sequence of points $ x_k \in \Sigma' $ such that $ x_k \rightarrow x $ holds $ \angle x_kxa \rightarrow 0 $. 
\end{definition}

\begin{remark}
If there exists a tangent line of the set $\Sigma$ at the point $x$ then there also exists a one-sided tangent of $\Sigma$ at $x$, but not vice versa. If there exists a one-sided tangent of the set $\Sigma$ at the point $x$ then there exists also a tangent ray of $\Sigma$ at $x$, but not vice versa.
\end{remark}


\color{black}

\begin{definition}
We will write $f(\rho)=o_\rho(1)$ if and only if $f(\rho)\rightarrow 0$ as $\rho \rightarrow 0$.
\end{definition}

\begin{definition}
For a given ray $(ax]$ and a number $\rho>0$ we will denote by a set of \textit{cusps}  $\mathbb{Q}_{\rho}^{a,x}$ the set consisting of such 
closed simply connected sets that $\forall Q \in \mathbb{Q}_{\rho}^{a,x}$ the following statements hold: 
\begin{itemize}
    \item $Q \subset \overline{B_{\rho}(x)}$;
    \item $Q \cap \partial B_{\rho'}(x)$ for each $0<\rho'\leq \rho$ is an arc (actually it is sufficient to state it only for $\rho'=\rho$)
    \item $(ax] \cap \overline{B_{\rho}(x)} \subset Q$;
    \item for an arbitrary sequence of points $z_k \in Q$ converging to $x$ holds $\angle z_k x a \rightarrow 0$.
The ray $(ax]$ will be called the \textit{axis of cusp $Q$}.

\end{itemize}
\end{definition}

\begin{remark}
Let $x\in E_\Sigma$ has a two-orientated energetic ray (id est there exists an energetic ray $(ix]$ orientated in the both directions). Then
\begin{enumerate}
    \item There exist two points $y_1, y_2 \in M$ such that  $\Sigma \cap B_r(y_1)=\Sigma \cap B_r(y_2) =\emptyset$, $|y_1x|=|y_2x|=r$, $x \in [y_1y_2]$ and $(ix) \perp (y_1y_2)$.
    \item For a sufficiently small $\rho>0$ there exist two cusps which union contains $\Sigma \cap B_\rho(x)$. One cusp is from set $Q^{i,x}_\rho$ and another one --- from set of cusps with axis, complementing ray $(ix]$ to line (ix). For example it can be the closures of the connected components of the set \[ \overline{B_\rho(x)}\setminus B_r(y_1) \setminus B_r(y_2) \setminus \{x\}.\]
    \item There is exactly one or two tangent rays of $\Sigma$ at $x$. If two then both of them should be contained in $(ix)$ (note that in this case there still can be only one energetic ray of $\Sigma$ at $x$). Anyway $\Sigma$ has a tangent line at the point $x$. 
\color{black}
\end{enumerate}

\label{star}
\end{remark}

\begin{lemma}
\begin{enumerate}
\item Let a set $W\ni x$ has tangent line $(ax)$ and moreover for each non  stabilized sequence of points $ x_k \in W$ such that $ x_k \rightarrow x $ holds $ \angle x_kxa \rightarrow 0 $. . Then for sufficiently small $\varepsilon>0$ there exists such cusp  $Q \in \mathbb{Q}_\varepsilon^{a,x}$ that $(W \cap \overline{B_\varepsilon(x)})\subset Q$ and $(W \cap \partial Q) \subset (\partial B_\varepsilon(x) \cup \{x\})$.
\label{kas} 
\item Let $\Sigma$ at point $x$ has exactly one energetic ray $(ax]$. Then for sufficiently small $\varepsilon>0$ there exists such cusp $Q \in \mathbb{Q}_\varepsilon^{a,x}$ that $G_\Sigma \cap  \overline{B_\varepsilon(x)} \subset Q$.
\label{en}
\item Let $\Sigma$ at point $x$ has two energetic rays $(ax]$ and $(bx]$.  
Then for sufficiently small $\varepsilon>0$ there exist such cusps $Q^a \in \mathbb{Q}_\varepsilon^{a,x}$ and $Q^b \in \mathbb{Q}_\varepsilon^{b,x}$, that  $(G_\Sigma \cap  \overline{B_\varepsilon(x)}) \subset Q^{a} \cup Q^{b}$.
\label{reb}
\item For each  $x\in E_\Sigma$ $\rho>0$ there exist such $\beta(\rho)>0$ that all energetic points from the ball $B_{\rho}(x)$ belong to the corner with opening $\beta(\rho)$ and with some energetic ray $(ax]$ as a bisector, where $\beta=o_\rho(1)$. Also it is easy to note that in view of compactness $E_\Sigma$ the value $\beta(\rho)$ doesn't depend on $x$.

 \label{BB}
\end{enumerate}
\label{con}
\end{lemma}
\begin{proof}
The proof of item~\ref{kas}. The existence of a one-sided tangent implies that there exists strictly decreasing sequence $r_k\rightarrow 0$ and non strictly decreasing sequence $s_k \rightarrow 0$ such that
\[
\forall z \in B_{\rho_k}(x) \cap W \text{ holds } \angle zxa <s_k.
\]
Denote 
\[
Q(\rho,s,x,a) \defeq \{z \big| |zx|\leq \rho, \angle zxa \leq s\}.
\]
Then the set \[
\bigcup_{k=1}^{\infty} \left(Q(\rho_k,s_k,x,a) \setminus B_{\rho_{k+1}}(x) \right)
\]
belongs to $\mathbb{Q}_{\rho_1}^{a,x}$.
The proof of item~\ref{en}. As there exits only one energetic ray at point $x$ for sufficiently small  $\varepsilon>0$ the line $(ax]$ is a tangent line for the set $G_\Sigma \cap B_\varepsilon(x)$ and the condition on item~\ref{kas} holds and condition of the statement~\ref{kas} holds. Then the previous item implies the statement.

The proof of item~\ref{reb} is similar and also follow from the definition of an energetic ray and item~\ref{kas}. 
The last item~\ref{BB} follows from remark~\ref{emptyball} item~\ref{netri},  items~\ref{en} and~\ref{reb} of this lemma and from the definition of cusps. 
\end{proof}

\begin{definition}
In the notations of the item~\ref{reb} of Lemma~\ref{con} an energetic point $\bar x$ from the set $G_\Sigma \cap Q^i$ will be called \textit{the energetic point of the energetic ray (ix]}, where $i=\{a,b\}$. In this case the ray $(ix]$ will be called \textit{the energetic ray, closest to $\bar x$}.
This definition doesn't depend of chosen cusps and have sense only in the neighbourhood of $x$. 
\end{definition}
\begin{lemma}
Let the intersection of a connected component of $\Sigma\setminus \{ x\}$ with the sufficiently small closed neighbourhood of $x$ be an arc and contain infinitely many energetic points of the energetic ray $(ix]$ and do not contain another energetic points. Then there exists a cusp with $(ix]$ as an axis, which contains this arc.
\label{dug_ok}
\end{lemma}
\begin{proof}
For sufficiently small $\rho>0$ the intersection of the connected  component of $\Sigma\setminus \{ x\}$ with a ball $\overline{B_\rho(x)}$ is an arc. Denote this arc by $\breve{[sx]}$. Let an injective function $f: [0;1] \rightarrow \R^2$ parameterizes the arc $\breve{[sx]}$. By the definition of an energetic ray the ray $(ix]$ is a one-sided tangent for the set $\breve{[sx]} \cap G_\Sigma$. I.e. for each sequence of energetic points  $\{g_j\}_{j=1}^\infty \subset \breve{[sx]} \cap G_\Sigma$ converging to $x$ the sequence of angles with the energetic ray $\angle g_jxi$ converges to $0$. But as energetic points forms a closed set, for each nonenergetic point $z \in \breve{[sx]}\setminus G_\Sigma$ there exists  such $c,d \in [0,1]$ that  $f(c),f(d) \in \breve{[sx]}  \cap G_\Sigma$, $f((c,d))\ni z$ and $f((c,d)) \subset \breve{[sx]}\setminus G_\Sigma$. Then the arc $f([c,d])$ should be a line segment. Thus $\angle zxi \leq \max (\angle f(c)xi, \angle f(d)xi)$.
So for each sequence of points $\breve{[sx]}$ converging to $x$ there exists the sequence of energetic points  $\breve{[sx]}$ converging to $x$  with larger angles with $(ix]$, still converging to $0$. Thus $(ix]$ is one-sided tangent for the whole $\breve{[sx]}$. Applying of the item~\ref{kas} Lemma~\ref{con} completes the proof.

\end{proof}

\begin{definition}
Let $M$ be a planar compact set. A connected set $\Sigma$ is called \textit{local minimizer} if $M \subset \overline{B_r(\Sigma)}$ and there exists an $\varepsilon >0$ such that for every connected $\Sigma_0$ such that $M \subset \overline{B_r(\Sigma_0)}$ and $\diam (\Sigma \triangle \Sigma_0) \leq \varepsilon$ one has $\H(\Sigma) \leq \H(\Sigma_0)$.
\end{definition}
\begin{remark}
As far as all arguments are local all statements of this work are true not only for maximal distance minimizers but also for local minimizers. 
\label{lcl}
\end{remark}

\section{Results} \label{res}

It is known (see~\cite{kuratowski2014topology}) that as $\Sigma$ has a finite length then the following lemma holds:


\begin{lemma}
Let the order of a point $ x $ in the set $ \Sigma $ is $ n $. Then for a sufficiently small $\varepsilon=\varepsilon(x) > 0$ the set $\Sigma \cap B_\varepsilon (x)$ contains $n$ arcs with one end at $\partial B_\varepsilon(x)$ and another at the $x$ intersecting only at the point $x$.
\label{ord_arcs}
\end{lemma}

In this work it is proved that $\Sigma \cap B_\varepsilon (x)$ coincides with these arcs. 

\subsection{The main theorem}

\begin{theorem2}
Let $\Sigma$ be a maximal distance minimizer for a compact set $M \subset \mathbb{R}^2$. Then \begin{itemize}
\item[(i)] $\Sigma$ is a union of a finite number of arcs (injective images of the segment $[0;1]$).
\item[(ii)] The angle between each pair of tangent rays at every point of $\Sigma$ is greater or equal to $2\pi/3$.
\item[(iii)] By (ii) the number of tangent rays at every point of $\Sigma$ is not greater than $3$. If it is equal to $3$, then there exists such a neighbourhood of $x$ that the arcs in it coincide with line segments and by (ii) the pairwise angles between them are equal to $2\pi/3$. 
\end{itemize}
\label{mainT1}
\end{theorem2}


\begin{proof}
Let $U\ni x$ be such a sufficiently small domain that
\[
\sharp (\partial U \cap \Sigma)=\oord_x\Sigma.  
\]
During the proof we will impose new restrictions on $U$ and it may be changed to the smaller one. Let $\Sigma'$ be a closure of a connected component of  $\Sigma \cap U \setminus \{x\}$. In view of the Corollary~\ref{oord_en} the number of such components is no more than three. Each component contains a point $x$ and certainly one point at the border $\partial U$. Without loss of generality the neighbourhood $U$ could be considered such small that every connected component of  $\Sigma \cap U \setminus \{x\}$ either doesn't contain any energetic points or contains infinitely many of them.  

If $\Sigma'$ doesn't contain any energetic points except $x$ then, in view of the Remark~\ref{stpt}, it is a line segment connecting $x$ with a point at the $\partial U$.

Let $\Sigma'$ contains infinitely many energetic points. Then for sufficiently small $U$ in view of the key Lemma~\ref{eva} all these energetic points are of the same energetic ray, denote it by $(ix]$. Thus in view of Lemma~\ref{uggi} it should be an arc and by Lemma~\ref{dug_ok} should be contained in some cusp with $(ix]$ as an axis. 

Then each connected component $\Sigma \cap U \setminus \{ x\}$ has one-sided tangent at the point $x$. Wherein, by the Lemma~\ref{angles2}, angles between these one-sided tangent is larger or equal $2\pi/3$. Then, if $x$ is an energetic point, there exist only one or two connected components   $\Sigma \cap U \setminus \{x\}$ (see~\ref{star} item~\ref{orientation}). In total, the set $\Sigma \cap U $ is one of three sets:
\begin{itemize}
    \item  A union of three line segments with one end at $x$ and another at $\partial U$, with pairwise angle $2\pi/3$ at the point $x$ (a regular tripod with center at $x$). 
    \item An arc with both ends at $\partial U$, which has two one-sided tangents at $x$, with angle greater or equal to $2\pi/3$ between them.
    \item  An arc, connecting $x$ and $\partial U$, having one-sided tangent at $x$
\end{itemize}
The theorem is proved as $\Sigma$ is a compact set. 
\end{proof}
\begin{corollary}
Let $\Sigma$ be a local minimizer for a compact set $M \subset \mathbb{R}^2$. Then, in view of remark~\ref{lcl}, if  \begin{itemize}
\item[(i)] $\Sigma$ is a union of a finite number of arcs (injective images of the segment $[0;1]$).
\item[(ii)] The angle between each pair of tangent rays at every point of $\Sigma$ is greater or equal to $2\pi/3$.
\item[(iii)] By (ii) the number of tangent rays at every point of $\Sigma$ is not greater than $3$. If it is equal to $3$, then there exists such a neighbourhood of $x$ that the arcs in it coincide with line segments and by (ii) the pairwise angles between them are equal to $2\pi/3$. 
\end{itemize}
\end{corollary}

We now turn to statements used in the proof of the theorem.

\subsection{Lemmas}
\begin{lemma}
Let $\Sigma$ have two one-sided tangents at $x$. Then the angle between these rays are at least $2\pi/3$. In particular they can not coincide.
\label{angles2}
\end{lemma} 

\begin{proof}
At first we need the following statements:
\begin{enumerate}
\item For a given $r>0$ and an arbitrary segment $[uv]$ of sufficiently small length $\varepsilon>0$ there exists such set $S^{uv}$ that  $|\H(S^{uv})|=o(\varepsilon)$ and $\overline {B_r([uv])} \subset \overline{B_r(S^{uv})}$. Moreover, $S^{uv}=S^u \cup S^v$, where $w\in S^w$ and $S^w$ is connected, with $w\in \{u,v\}$.   
\begin{proof}
Let $C$ be a point equidistant from $u$ and $v$, located at a distance of $r$ from the straight line $(uv)$. And let $C_w$ be such points that $C_w \in [wC]$ and $|C_wC|=r$ with $w\in \{u,v\}$. In addition let the points $C_u^-$ and $C_v^-$ be symmetric reflections of the points $C_u$ and $C_v$ with respect to the line $(uv)$ (see Fig.~\ref{Suv}).
 \definecolor{ffqqqq}{rgb}{1.,0.,0.}
\definecolor{ttzzqq}{rgb}{0.2,0.6,0.}
\definecolor{uuuuuu}{rgb}{0.26666666666666666,0.26666666666666666,0.26666666666666666}
\definecolor{ududff}{rgb}{0.30196078431372547,0.30196078431372547,1.}
\begin{figure}
\begin{tikzpicture}[line cap=round,line join=round,>=triangle 45,x=1.0cm,y=1.0cm]

\clip(-3.6987668539577405,-1.732942517054489) rectangle (3.406482070449608,3.7497281006907395);
\draw[line width=0.8pt,color=ffqqqq,fill=ffqqqq,fill opacity=0.10000000149011612] (12.,-2.) -- (13.,-2.) -- (13.,-2.);
\draw [line width=2.pt,dash pattern=on 6pt off 6pt] (-2.,0.) circle (4.cm);
\draw [line width=2.pt,dash pattern=on 6pt off 6pt] (2.,0.) circle (4.cm);
\draw [line width=2.8pt,color=ttzzqq] (-2.,0.)-- (2.,0.);
\draw [line width=2.pt,dash pattern=on 6pt off 6pt] (0.,3.464101615137755) circle (3.464101615137755cm);
\draw [line width=2.pt,dash pattern=on 3pt off 3pt] (-2.,0.)-- (0.,3.464101615137755);
\draw [line width=2.pt,dash pattern=on 3pt off 3pt] (2.,0.)-- (0.,3.464101615137755);
\draw [line width=2.4pt,color=ffqqqq] (-2.,0.)-- (-1.7320508075688772,0.4641016151377546);
\draw [line width=2.8pt,color=ffqqqq] (2.,0.)-- (1.7320508075688772,0.4641016151377546);
\draw [line width=2.4pt,color=ffqqqq] (-2.,0.)-- (-1.7320508075688772,-0.4641016151377546);
\draw [line width=2.8pt,color=ffqqqq] (2.,0.)-- (1.7320508075688772,-0.4641016151377546);
\draw [line width=2.pt,dash pattern=on 6pt off 6pt] (15.,2.) circle (4.cm);
\draw [line width=2.pt,dash pattern=on 6pt off 6pt] (15.,-6.) circle (4.cm);
\draw [line width=2.4pt,color=ffqqqq] (12.218925567339127,-0.875)-- (12.421875,-0.6651963313000041);
\draw [line width=2.8pt,color=ffqqqq] (15.,-2.)-- (15.,-1.7080992435478317);
\draw [line width=2.4pt,color=ffqqqq] (12.218925567339127,-0.875)-- (12.218925567339127,-1.1669007564521667);
\draw [line width=2.8pt,color=ffqqqq] (15.,-2.)-- (14.797050567339124,-2.209803668699993);
\draw [line width=2.4pt,color=ffqqqq] (12.218925567339127,-3.125)-- (12.218925567339125,-2.8330992435478346);
\draw [line width=2.8pt,color=ffqqqq] (15.,-2.)-- (14.797050567339124,-1.7901963313000036);
\draw [line width=2.4pt,color=ffqqqq] (12.218925567339127,-3.125)-- (12.421874999999996,-3.334803668699996);
\draw [line width=2.8pt,color=ffqqqq] (15.,-2.)-- (15.,-2.291900756452172);
\draw [shift={(15.,-2.)},line width=2.pt,color=ffqqqq]  plot[domain=2.757195879094154:3.525989428085432,variable=\t]({1.*3.*cos(\t r)+0.*3.*sin(\t r)},{0.*3.*cos(\t r)+1.*3.*sin(\t r)});
\draw [line width=2.8pt,color=ffqqqq] (12.,-2.)-- (15.,-2.);

\begin{scriptsize}

\draw [fill=uuuuuu] (0.,3.464101615137755) circle (2.0pt);
\draw[color=uuuuuu] (0.34915595533223287,3.45987191990982186) node {$C$};

\draw [fill=ffqqqq] (-1.7320508075688772,0.4641016151377546) circle (2.0pt);
\draw [fill=red] (1.7320508075688772,0.4641016151377546) circle (2.0pt);
\draw [fill=ffqqqq] (-1.7320508075688772,-0.4641016151377546) circle (2.0pt);
\draw [fill=ffqqqq] (1.7320508075688772,-0.4641016151377546) circle (2.0pt);

\draw[color=ffqqqq] (-1.4894097830443668,0.2815381922731308) node {$C_u$};
\draw[color=ffqqqq] (2.2988196905990326,0.5837966819709189) node {$C_v$};
\draw[color=ttzzqq] (-2.2,0) node {$u$};
\draw[color=ttzzqq] (2.2,0) node {$v$};
\draw [fill=ttzzqq] (-2,0) circle (2.0pt);
\draw [fill=ttzzqq] (2,0) circle (2.0pt);

\draw[color=ffqqqq] (-1.4894097830443668,-0.2815381922731308) node {$C^-_u$};
\draw[color=ffqqqq] (2.2988196905990326,-0.5837966819709189) node {$C^-_v$};

\end{scriptsize}
\end{tikzpicture}

\caption{Constructing of the set $S^{uv}$.}
\label{Suv}
\end{figure}
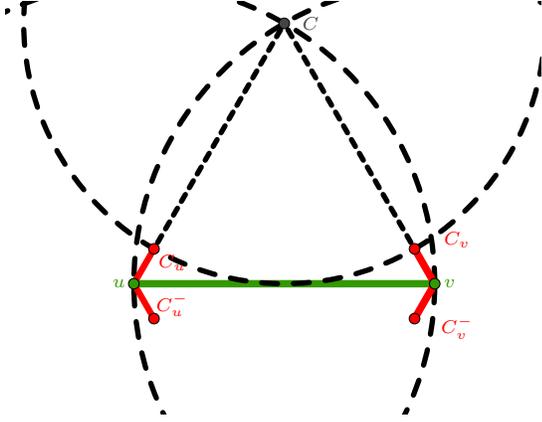
Then we can define $S^u \defeq [uC_u] \cup [uC_u^-]$ and $S^v \defeq [vC_v] \cup [vC_v^-]$. It is easy to see that these sets meet all conditions.
\end{proof}
\item For an arbitrary cusp $Q\in Q_\varepsilon^{a,x}$ there exists such a set $S^Q$ that $|\H(S^Q)|=o(\varepsilon)$ and $\overline{B_r(Q)} \subset \overline{B_r(S^Q)}$. Moreover, $S^Q=S_1 \cup S_2$, where $S_1 \supset (\partial Q \cap \partial B_\varepsilon(x))$, $x \in S_2$ and both $S^1$ and $S^2$ are connected (see Fig.~\ref{TQ}).
\begin{proof}
Let points $T_1, T_2 \in \partial B_\varepsilon(x)$ be such that $\angle T_1xa = \max_{z\in Q} \angle zxa$ and $T_2$ be a reflection of $T_1$ with respect to $(ax)$. And let $\breve{T_1T_2}$ be the smallest arc of the $\partial B_\varepsilon(x)$. Then, since $ Q $ lies in the figure bounded by $ \breve {T_1T_2} \cup [T_1x] \cup [T_2x] $ the following statement is true:  
\[
\overline{B_r(Q)} \subset \overline{B_r(\breve{T_1T_2}\cup [T_1x] \cup [T_2x])}.
\]
It is easy to see that the set
\[
S^Q \defeq \breve{T_1T_2} \cup S^{T_1x} \cup S^{T_2x}
\]
meets all requirements.
\end{proof}
 \item For a cusp $Q\in Q_\varepsilon^{a,x}$ there exists such connected set $T^Q$ that $|\H(T^Q)|=\varepsilon+o(\varepsilon)$, $\overline{B_r(Q)} \subset \overline{B_r(T^Q)}$ and $T \supset ((\partial Q \cap \partial B_\varepsilon(x)) \cup \{x\})$ (see Fig.~\ref{TQ}).
 \begin{proof}
 In the notations of the previous item let us define $Z \defeq (ax] \cap \partial B_\varepsilon(x)$. Then the set 
 \[
 T^Q \defeq S^Q \cup [Zx].
 \]
 \end{proof}
meets all requirements.
\end{enumerate}

\begin{figure}
\center{\includegraphics[width=0.4\linewidth]{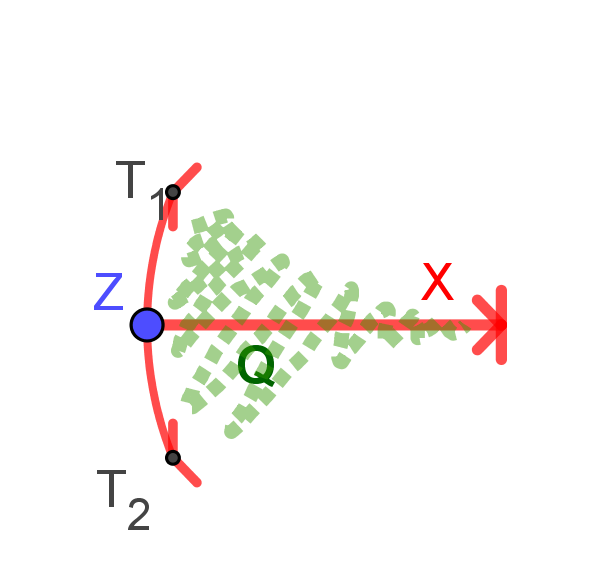}}
\caption{Construction of the sets $S^Q$, $T^Q$.}
\label{TQ}
\end{figure}

In view of Lemma~\ref{con} i.~\ref{kas} there exists such $\varepsilon_0>0$, 
that for an arbitrary $\varepsilon<\varepsilon_0$ inside the ball $B_\varepsilon(x)$ each of these two components lie inside a cusp with its one-sided tangent as an axis.  
Define the closures of intersections of connected components with $B_\varepsilon(x)$ by $W_1$ and $W_2$. Clearly 
\[
\H(W_i)\geq \varepsilon
\]
for $i=\{1,2\}$.

Assume that these two connected components have the same one-sided tangent $(ax]$. Then in view of i.~\ref{kas} Lemma~\ref{con} for $W=W_1 \cup W_2$ there exists such containing them cusp $Q \in  Q_\varepsilon^{a,x}$ such that $\partial Q \cap (W_1 \cup W_2) \subset \partial B_\varepsilon(x)$.

Then the set 
\[
\Sigma \setminus (W_1 \cup W_2) \cup T^Q
\]
is better than $\Sigma$: it is connected, has a shorter length and not larger energy than $\Sigma$. This contradicts the assumption that two connected components can have the same one-sided tangent line.

Let us now show that the angle between one-sided tangents is at least $ 2 \pi / 3 $. Assume the opposite, that is, that the angle between the one-sided tangents $ (ax] $ and $ (bx] $ (of respectively $W_1$ and $W_2$) is strictly less than $ 2 \pi / 3 $. Denote $A \defeq (ax] \cap \partial B_\varepsilon(x)$ and $B \defeq (bx] \cap  \partial B_\varepsilon(x)$. Then, by Proposition~\ref{angles4}, the set
\[
\Sigma \setminus (W_1 \cup W_2) \cup S^{Q_a}\cup S^{Q_b} \cup St(A,B,x)
\]
is connected, has a shorter length and not larger energy than $\Sigma$. Contradiction.

\end{proof}
\begin{corollary}
Let $\Sigma \setminus x$ have one-sided tangent $(ax]$. Then for sufficiently small $\varepsilon>0$ the statement $\H(W \cap B_\varepsilon(x))=\varepsilon+o(\varepsilon)$ holds, where $W$ is a closure of connected component of $\Sigma \setminus \{ x\}$. As in other case the set $\Sigma \setminus (B_\varepsilon(x) \cap W) \cup T^Q$ for some containing $B_\varepsilon(x)\cap W$ cusp $Q\in Q^{a,x}_\varepsilon$ is better than $\Sigma$: it is connected, has a shorter length and not larger energy than $\Sigma$. A contradiction.
\label{len0}
\end{corollary}


\begin{proposition}
For an arbitrary point $x\in S_\Sigma \sqcup X_\Sigma$ there exists such $\varepsilon=\varepsilon(x)>0$ that the set $ \Sigma \cap B_\varepsilon (x) $ consists of $ \oord_x \Sigma $ line segments with one end at $ x $ and another one at $\partial B_\varepsilon(x)$, with pairwise angles between them at least $ 2 \pi / 3 $. Moreover, in the case of $ x \in S_ \Sigma $, the set $ \overline{B_\varepsilon (x)} \cap \Sigma $ is a segment or a regular tripod centered at $ x $, and in the case of $ x \in X_\Sigma $ --- a segment with an end at $ x $ and another one at $\partial B_\varepsilon(x)$, or the union of two segments with one end at $ x $ and another one at $\partial B_\varepsilon(x)$, with angle at least $ 2 \pi / 3 $ between them.

\label{discr_str}
\end{proposition}

\begin{proof}
It is easy to see that for each $x \in \Sigma$ the inequality $\oordball_x \Sigma \leq 6$ holds. Assume the contrary, then for sufficiently small $\varepsilon>0$ the co-area inequality entails $\H(\Sigma \cap \overline{B_\varepsilon (x)}) \geq 7 \varepsilon$. But then the set $\Sigma' \defeq \Sigma \setminus B_\varepsilon(x) \cup \partial B_\varepsilon(x)$ is better than $\Sigma$: it is connected, its energy is not greater and its length is strictly less than of $\Sigma$. This gives a contradiction with the optimality of $\Sigma$.

Now let $x \in S_\Sigma \sqcup X_\Sigma$ and let $\varepsilon>0$ be such that $\Sigma \cap \overline{B_\varepsilon(x)} \setminus \{x\} \subset S_\Sigma$. Thus, there exists an infinite monotonically decreasing sequence $\varepsilon_i \rightarrow 0$ such that $\sharp (\partial B_{\varepsilon_i}(x) \cap \Sigma) = \oordball_x \Sigma \leq 6$ for every $i$. Without loss of generality let $\varepsilon_1<\varepsilon$. For numbers $l>k$ let us define the set
\[
\Sigma_{k,l} \defeq \Sigma \cap \overline{{B}_{\varepsilon_k}(x)} \setminus B_{\varepsilon_l}(x).
\]
Note that $\Sigma_{k,l} \subset S_\Sigma$ is a Steiner forest and all its vertices belong to $\partial B_{\varepsilon_k}(x) \cup \partial B_{\varepsilon_l}(x)$. Since $\Sigma_{k,l}$ has not more vertices than $13$, the number of its branching points is at most $11$ for each $l>k$ and $\Sigma_{k,l}$ consists of a finite number of line segments, less than $25$. Then there exists such a number $m>k$ that the set $\Sigma \cap \overline{B_{\varepsilon_m}(x)}$ does not contain a branching points or a maximal segment not containing the point $x$. So the number of segments in $\Sigma \cap \overline{B_{\varepsilon_m}(x)}$ is at most the number of segments of every set $\Sigma_{m,n}$, $n>m$ (which is not greater than $24$); in particular it is finite.    

In view of Lemma~\ref{angles2} the set $\Sigma \cap B_{\varepsilon_m}(x)$ should be either a segment with one end at the point $x$ and another at $\partial B_{\varepsilon_m}(x)$, or the onion of two such segments with the angle at least $2\pi/3$ between them, or the union of three such segments with angles $2\pi/3$ between them (i.e. regular tripod with the center at $x$). 
\end{proof}

Thus, it remains for us to deal with the behavior of $\Sigma$ in the neighborhoods of points of the set $ E_\Sigma $.

\begin{lemma}
The inequality $\oordball_x \Sigma\leq 3$ holds for every point $x \in \Sigma$.

\label{ord_en}
\end{lemma}

\begin{proof}
For a point $ x \in X_\Sigma \cup S_\Sigma $ the statement follows trivially from Proposition~\ref{discr_str}. Now let $ x \in G_\Sigma $ and $ \oordball_x \Sigma \geq 4 $. Then, in view of the co-area inequality for sufficiently small $\varepsilon>0$ one has $ \H (B_\varepsilon (x) \cap \Sigma) \geq 4 \varepsilon $. On the other hand, for a sufficiently small $ \varepsilon> 0 $, the length $ \H (\partial B_ \varepsilon (x) \setminus B_r (y)) <3.5 \varepsilon $, where $ y \in M $ is an arbitrary point corresponding to the point $ x $. But then the set 
\[
(\Sigma \setminus B_\varepsilon(x) \cup \partial B_\varepsilon(x) \setminus B_r(y)) \cup S^{T_1T_2},
\]
where $T_1, T_2 = \partial B_\varepsilon(x) \cap \partial B_r(y(x))$, and the set $ S ^ {T_1T_2} $ is defined in the proof of Lemma~\ref{angles2},
is better than the set $\Sigma $. This contradiction completes the proof of the lemma.
\end{proof}
Then in view of the Remark~\ref{rem_ord} the following corollary is true: 
\begin{corollary}
For every point $x\in \Sigma$ the inequality $\oord_x \Sigma \leq 3$ holds.
\label{oord_en}
\end{corollary}
\begin{lemma} 
Let $x\in E_\Sigma$, $\rho>0$, 
 $\bar{x}\in B_{\rho}(x)\cap G_\Sigma \setminus \{x\}$  --- an energetic point of a ray $(ix]$,
$\varepsilon<\min(|\bar x x|, \rho-|\bar x x|)$ and let point  
$w\in B_\varepsilon(\bar{x})$ be such that 
$(w\bar{x}) \parallel (ix)$, or $w=\bar{x}$. Then there exists such connected set $S=S(\bar{x},w,x,\varepsilon) \ni w$, that 
\[
F_M(\Sigma \setminus B_\varepsilon(\bar{x}) \cup S) \leq F_M(\Sigma)
\]
and $\H(S)=\varepsilon o_\rho(1)$.
\label{Losiu}
\end{lemma}

Note that the set $\Sigma \setminus B_\varepsilon(\bar{x}) \cup S$  does not have to be connected.

Usually we will use the following formulation of this lemma:
\begin{multline*}
 \forall c>0 \ \forall x \in E_\Sigma \  \exists \rho_0>0 \ \forall \rho<\rho_0 \ \forall \bar{x}\in B_\rho(x)\cap G_\Sigma \setminus \{x\} \  \exists \varepsilon_0>0:\\ \forall \varepsilon<\varepsilon_0 \  \forall w \in B_\varepsilon(\bar{x}): (w\bar{x}) \parallel (ix) \vee (w=\bar x) \  \exists \mbox{ connected set } S\ni w:\\ F_M(\Sigma \setminus B_\varepsilon(\bar{x}) \cup S) \leq F_M(\Sigma) \mbox{ and } \H(S)<c\varepsilon.
\end{multline*}
Also, in view of compactness of $E_\Sigma$ the second and third predicates can be swapped: \[\forall c>0 \ \exists \rho_0>0 \  \forall x \in E_\Sigma \ \dots.\] 

\begin{proof}


In the further proof, the point $ x \in E_\Sigma $ will be assumed to be fixed; all points and quantities will be assumed to correspond to the conditions of the lemma. For a positive number $ \varepsilon $ and a point $ \bar x $, define the set

\[
U(\bar{x},\varepsilon) \defeq \left \{ y\in M | \dist(y, \overline{B_{\varepsilon}(\bar{x})})\leq r, \dist(y, \Sigma \setminus \overline{B_{\varepsilon}(\bar{x})})> r \right \}.
\]

\begin{enumerate}
    \item If an energetic ray $(ix]$ is one-orientated, then for a sufficiently small $\rho>0$ and $\varepsilon>0$ the set $U(\bar{x},\varepsilon)$ is located in one half-plane bounded by $ (ix) $(the one into which the ray is orientated). 
    
    If an energetic ray $(ix]$ is two-orientated for each $\bar x$ either exist such 
     $\varepsilon>0$, that the set $U(\bar{x},\varepsilon)$ is located in one half-plane bounded by $ (ix) $
     or in both half-planes there exist different points $y (\bar {x}) \in M$ corresponding to $ \bar{x} $. In what follows, we will assume that $\rho> 0 $ and $ \varepsilon> 0 $ satisfy these conditions.
    \label{1115} 
    \begin{proof}
    The case of two-orientated $(ix]$ is clear.
    Let $(ix]$ be one-orientated. Without loss of generality we can consider $(ix]$ as a negative part of the abscissa, oriented to upper half-plane, i.e. points corresponding to energetic points of ray $(ix]$ lie upper abscissa. 
    Assume the contrary to the statement being proved, i.e. for each $\rho>0$ and $\varepsilon_0>0$ there exists such positive $\varepsilon<\varepsilon_0$ and such energetic point $\bar x \in E_\Sigma \cap B_\rho(x) \setminus \{ x\}$ that the intersection of $U(\bar{x}, \varepsilon)$ with lower half-plane is not empty. So for a point $\bar x$ there exists a sequence $\{\varepsilon_k\}$ monotonically decreasing to $0$ such that each set $U(\bar{x}, \varepsilon_k)$ contains some points of the lower half-plane, denote an arbitrary such point in each set by $u_k$. Note that $U(\bar{x}, \varepsilon_k)\supset U(\bar{x}, \varepsilon_{k+1})$ as $\varepsilon_k>\varepsilon_{k+1}$. Let $u = \lim_{k \rightarrow \infty} u_k$. As $M$ is a closed set $u\in M$. Moreover as $|u_k\bar x|\leq r+\varepsilon_k$, the inequality $|u\bar x| \leq r$ holds. 
    Assume $B_r(u)\cap \Sigma$ is not empty, thus, in view of connectivity of $\Sigma$, the intersection contains a point $x_1 \neq \bar{x}$. Consider such number $L$ that for every $l\geq L$ the inequalities $\varepsilon_l< |x_1\bar x|$ and $|u_lu|<r-|ux_1|$ hold. But then $x_1 \notin B_{\varepsilon_l}(\bar x)$ and $|\dist(x_1,U(\bar{x},\varepsilon_l))|<r$, which contradicts to the definition of the set $U(.,.)$. Thus $B_r(u)\cap \Sigma$ is empty, which means that $u$ corresponds to $\bar x$. Then there exists a sequence of energetic points of the ray $(ix]$ converging to $x$ with a points of $M$, lying in a lower half-plane, corresponding to them. The contradiction to the definition of one-orientated energetic ray.

    \end{proof}
        \begin{remark}
Obviously, $y(\bar{x}) \in \overline{U(\bar{x},\varepsilon)}$, where $y(\bar{x})\in M$ is an arbitrary point, corresponding to $\bar{x}$.
    \label{rem_c}
    \end{remark}
    
    \item Clearly, as $\Sigma \cap B_\varepsilon(\bar{x}) \subset B_\varepsilon(\bar{x}) \setminus B_r(y(\bar{x}))$, the inequality 
    $\angle y(\bar{x})\bar{x}s>\pi/2-o(\varepsilon)$ holds for every $s \in B_\varepsilon(\bar{x}) \cap \Sigma$ and $y(\bar{x})$, corresponding to $\bar{x}$. In other words
    \[
    \forall c>0 \  \exists \rho>0 \ \exists \varepsilon>0: \forall \bar x \in B_\rho(x) \ \forall s \in B_\varepsilon(\bar x) \quad \angle y(\bar x) \bar x s >\frac{\pi}{2}-c\varepsilon.
    \]
    
    \label{113}
    \item For each $u\in U(\bar{x},\varepsilon)$ the following statement is true: $|\angle((u\bar{x}),(ix))-\frac{\pi}{2}|=o_\rho(1)$. 

    \label{1145}
    \begin{proof}
Assume the contrary. Then there exists such number $c>0$, sequence  $\{x_k\}$ of energetic points of the energetic ray $(ix]$ converging to $x$, sequence $\varepsilon_k>0$ decreasing to $0$ and sequence of points $u_k \in U(x_k, \varepsilon_k)$ such that $|\angle((u_k x_k),(ix))-\frac{\pi}{2}|>c$ and, moreover,  $B_{\varepsilon_l}(x_l)\cap B_{\varepsilon_m}(x_m)=\emptyset$ for different numbers $l$, $m$ and $\varepsilon_k < \frac{|x_k x|}{2}$. Note that $x \notin B_{\varepsilon_k}(x_k)$, and thus by the definition of $U$ for each number $k$ the inequality $\dist(x, U(x_k,\varepsilon_k))>r$ holds. 

In this case, at least one of two options is realized:\begin{itemize}
\item [I:] There exists a subsequence $\{I_{j=1}^\infty\}$ of numbers
such that $\angle([u_{I_k} x_{I_k}),(ix])<\frac{\pi}{2}-c$ for each number $k$. As $\{x_{I_k}\}$ are energetic points of an energetic ray $(ix]$ converging to $x$ there exists such number $N$ that for every number $\forall k >N$ the inequalities $\angle ixx_{I_k}<c/4$ and $\angle x_{I_k}u_{I_k}x<c/4$  hold. Thus $\angle u_{I_k} x_{I_k} x\leq \angle([u_k x_{I_k}),(ix])+\angle ixx_{I_k}<\frac{\pi}{2}-\frac{3c}{4}$. Then
\[
|u_{I_k}x|=\frac{|u_{I_k}x_{I_k}| \sin \angle u_{I_k}x_{I_k}x}{\sin \angle x_{I_k}xu_{I_k}}<\frac{(r+\varepsilon_k) \sin(\frac{\pi}{2}-\frac{3c}{4})}{\sin(\frac{\pi}{2}+\frac{c}{2})}=\frac{(r+\varepsilon_k) \cos \frac{3c}{4}}{\cos \frac{c}{2}}<r,
\]
where the last inequality is true for a sufficiently large  $ k $ (depending on $ c $), since $ c> 0 $ can be considered less than $ \pi / 2 $. This contradicts to $\dist(x, U(x_{I_k},\varepsilon_{I_k}))>r$.

\item [II:] There exists a subsequence $\{J_{j=1}^\infty\}$ of numbers such that $\angle ([u_{J_k} x_{J_k}),(ix])>\frac{\pi}{2}+c$. Let us consider a limit of the subsequence: $u \defeq \lim_{k \rightarrow \infty} u_{J_k}$. Then $|ux|\leq r$ (and hence $|ux|=r$) and $\angle uxi \leq\pi/2-c$. 

 As $\{x_{J_k}\}$ are energetic points of an energetic ray $(ix]$ converging to $x$ there exists such number $N_1$ that for every number $\forall k >N_1$ the inequalities $\angle ixx_{J_k}<c/4$ (and, therefore $\angle uxx_k<\pi/2 - 3c/4$) and $\angle x_{J_k}ux\leq c/4$ (for $u=\lim_{l \rightarrow \infty} u_{J_l}$)  hold. Then
 \[
|x_{J_k}u|=\frac{\sin \angle x_{J_k}xu}{\sin \angle ux_{J_k}x}|ux|<\frac{\cos \frac{3c}{4}}{\cos{\frac{c}{2}}}r.
\]
There also exists such $N_2$ that for every $m\geq N_2$ the inequality  $|u_{m}u|<r \left (1-\frac{\cos \frac{3c}{4}}{2\cos{\frac{c}{2}}}\right)$ holds. Then for different $J_k\neq m$ with $k>N_1$ the following is true: \[|u_{m}x_{J_k}|< |x_{J_k}u|+|u_{m}u|\leq \frac{\cos \frac{3c}{4}}{\cos{\frac{c}{2}}}r + \left (1-\frac{\cos \frac{3c}{4}}{2\cos{\frac{c}{2}}}\right) < r.\]
We have a contradiction with $x_{J_k} \notin B_{\varepsilon_m}({x_m})$.
 
\end{itemize}
    \end{proof}
    
\begin{remark}
       In view of Remark~\ref{rem_c} also $|\angle((y(\bar{x})\bar{x}), (ix))-\frac{\pi}{2}|=o_\rho(1)$ holds. 
       \label{rem_cc}
\end{remark}
 
    \item  Clearly that for every $x_1 \in B_\varepsilon(\bar{x})$ the inequality $\dist(x_1, U(\bar{x},\varepsilon)) \geq r-\rho$ holds: otherwise $\dist(x, U(\bar{x},\varepsilon)) \leq r$, which is impossible as $\varepsilon<|\bar{x}x|$ which entails $x \in \Sigma \setminus \overline{B_\varepsilon(\bar x)}$. 
    \label{777}
    \item Let  $\Lambda((nl), \gamma) \defeq \{z \ | \ \angle((zl),(ln)) \leq \gamma \}$. Then $U(\bar{x},\varepsilon) \subset \Lambda({(y(\bar{x}),w),o_\rho(1)})$ where $w$ is defined in the statement of the proven lemma.

    \begin{remark}
    Hereinafter, in case of $U(\bar{x},\varepsilon)$ located in one half-plane relative to $(ix)$, we could use instead of a set $\{z \ | \ \angle((zl),(ln)) \leq \gamma \}$ --- a set $ \{z \ | \ \angle zln \leq \gamma \}$.
    \end{remark}

   \begin{proof}
It is sufficiently to prove that for every $u\in U(\bar x, \varepsilon)$ the equality $\widehat{(y(\bar x)w),(uw)}=o_\rho(1)+\pi k$ holds for some $k\in \mathbf{Z}$.
In this item the angle between lines is supposed to be counterclockwise orientated and belongs to $[-\pi/2,pi/2]$. Then for every three lines $l_1$, $l_2$ and $l_3$ the equality $\widehat{l_1,l_2}+\widehat{l_2,l_3} = \widehat{l_1,l_3}+\pi k $ for some $k\in \mathbf{Z}$.

Applying item~\ref{777} to $x_1 \defeq \bar{x}$ 
   we get $\dist(\bar{x}, U(\bar{x},\varepsilon)) \geq r-\rho$. Then consider the triangle $\triangle u\bar x w$: 
\[
\sin (\angle \bar x u w) =  \frac{|\bar x w|}{|\bar x u|}\sin (\angle \bar x w u)\leq \frac{\varepsilon}{r-\rho}=O(\varepsilon)=o_\rho(1). 
\]
Thus $\angle((uw),(u\bar x))=|\angle \bar x u w| =o_\rho(1)$ and particularly $\angle ((y(\bar x)w),(y(\bar x) \bar x))=\angle \bar x y(\bar x) w =o_\rho(1)$. 

So in view of item~\ref{1145} and Remark~\ref{rem_cc} the equality
\begin{equation*}
\widehat{(y(\bar x) w),(uw)}=\widehat{(y(\bar x)w),(y(\bar x) \bar x)}+\widehat{(y(\bar x) \bar x),(ix)}+ \widehat{(ix),(u \bar x)}+ \widehat{(u\bar x),(uw)}+ \pi l=o_\rho(1)+\pi k
\label{ang} 
\end{equation*} holds for some $l, k\in \mathbf{Z}$.



   
   \end{proof}

    \label{115}
    \color{black}
    \item For every  $u\in U(\bar{x}, \varepsilon)$ there exists such $v_1$ (does not have to belong to $\Sigma$) that $v_1 \in B_\varepsilon(\bar{x})$, $|v_1u|\leq r$, $\angle ((v_1\bar{x}),(ix))=o_\rho(1)$.

    \begin{proof}
    For every $u\in U(\bar{x}, \varepsilon)$ there exists such $v \in \Sigma \cap B_\varepsilon(\bar{x})$ that $|vu| \leq r$. Let $y(\bar x)$ be a point of $U(\bar x, \varepsilon)$ corresponding to $\bar x$ and located at the same half-plane as $u$ with respect to $(ix)$.
    Define $v_1 \defeq v$ in case that $\pi/2> \angle y(\bar{x})\bar{x}v$. Note that $\angle y(\bar{x})\bar{x}v>\pi/2-o_\varepsilon(1)$ by item~\ref{113}. And in the case $\angle y(\bar{x}) \bar{x} v \geq \pi/2$ one can define $v_1 \in [vu]$ such way that $(v_1\bar x) \perp ( x y(\bar x))$.

    \end{proof}

    \label{117}
    \item The inequality $|wu|\leq r+o(\varepsilon)+\varepsilon o_\rho(1)$ holds. Thus $U(\bar{x},\varepsilon) \subset B_{r+\varepsilon o_\rho(1)}(w)$.
    \label{118}
    \begin{proof}
Note that by items~\ref{1145} and~\ref{117} lines $(u\bar x)$, $(uw)$, $(uv_1)$ are almost perpendicular to $(ix)$ and lines $(\bar xw)$, $(\bar xv_1)$ are almost parallel to $(ix)$.  Id est $|\angle((ux),(ix))-\pi/2|, |\angle((uw),(ix))-\pi/2|, |\angle((uv_1),(ix))-\pi/2|=o_\rho(1)$
and $\angle((\bar x w),(ix)), \angle((\bar x v_1),(ix))=o_\rho(1)$.
Thus $|\angle uv_1\bar x-\pi/2|=o_\rho(1)$ and $|\angle u\bar xw-\pi/2|=o_\rho(1)$. 

Then applying law of cosines to the triangle $\triangle uv_1 \bar x$ we get $|u\bar x|^2=|uv_1|^2+|v_1\bar x|^2-2|uv_1||v_1\bar x| \cos \angle uv_1\bar x$. Which entails $|u\bar x|^2\leq r^2+\varepsilon^2+o_\rho(1)\varepsilon$,
so $|u\bar x|\leq r+o_\rho(1)\varepsilon$.

Now by law of cosines for the triangle $\triangle uw\bar x$: $|uw|^2=|u\bar x|^2+|w\bar x|^2-2|u\bar x||w\bar x|\cos(\angle u\bar xw)$.
Then $|uw|^2\leq ( r+o_\rho(1)\varepsilon)^2+\varepsilon^2+\varepsilon o_\rho(1)$, which implies $|uw|\leq r+o_\rho(1)\varepsilon$.

    \end{proof}
    \item
 Clearly, $U(\bar{x},\varepsilon) \subset \Lambda({(y(\bar{x}),w),o_\rho(1)})\cap B_{r+o_\rho(1)\varepsilon}(w)$.
    \label{119}
    \color{black}
    \item There exists a connected set $S\ni w$ such that $\H(S)=\varepsilon o_\rho(1)$ and 
    \[
    (\Lambda ((y(\bar{x})w), \gamma) \cap B_{a}(w))\subset \overline{B_r(S)},
    \]
where $\gamma=o_\rho(1)$, $a=r+\varepsilon o_\rho(1)$.

    \label{120}
\begin{proof}
Let us construct $S$ in the following way: 
\[ 
S \defeq \left(( (wy(\bar x)) \cap B_{a-r}(w)) \cup \partial  B_{a-r}(w))\right)\cap \Lambda ((y(\bar{x})w), \gamma).
\] 
$S$ obviously satisfies both conditions.

\end{proof}    
\begin{remark}
It is easy to see that in case $U(\bar{x},\varepsilon)$ is located in one half-plane with respect to $(ix)$ (and hence with respect to a line $l_w$ containing $w$ and parallel to $(ix)$) the set $S$ can be chosen half as long: only one part of two symmetrical about $l_w$ is sufficient.
\end{remark}
\end{enumerate}
Items~\ref{119}, \ref{120} and the definition of the set $U(\bar{x},\varepsilon)$  imply the statement of the lemma. 
\end{proof}

Hereinafter we will write $o_\rho(1)$-almost, if something is true to within $o_\rho(1)$. 

\begin{corollary}
Item~\ref{1145} of Lemma~\ref{Losiu}, in view of the Remark~\ref{emptyball}, entails that the energetic rays of an arbitrary point $\bar{x} \in E_\Sigma \cap B_\rho(x) $ are $o_\rho(1)$-almost parallel to $(ix)$, id est $|\angle((ix),l_{\bar x})=o_\rho(1)|$ where $l_{\bar x}$ is a line containing an energetic ray of the set $\Sigma$ with a vertex at the point $ \bar x$.
\label{alm_parr}
\end{corollary}

\begin{lemma} For an arbitrary $x\in E_\Sigma$ there exists such $\rho>0$, that for every $\bar{x}\in B_{\rho}(x)\cap G_\Sigma \setminus \{x\}$ the inequality $\oordball_{\bar{x}}\Sigma \leq 2$ holds. 
\label{ordball_}
\end{lemma}

\begin{proof}
Let $n \defeq \oordball_{\bar x}\Sigma$ and let $\varepsilon>0$ be such a sufficiently small number that 
\[
\sharp (\partial B_{\varepsilon}(\bar x) \cap \Sigma) = n. 
\]
Then, in view of the definition of the circular order and of co-area inequality the following is true: $\H(B_\varepsilon(\bar x)\cap \Sigma) \geq n \varepsilon$. 
As $\bar{x}$ is an energetic point, these intersection is empty $B_r(y(\bar x)) \cap \Sigma =\emptyset$. So all $n$ points of  $\Sigma \cap \partial B_\varepsilon(\bar x)$ are located at the arc with size $\pi+o(\varepsilon)$.

Let $n \geq 4$. There exists such $\varepsilon_0$ that for every positive number $\varepsilon<\varepsilon_0$ the inequality $\H(\partial B_\varepsilon(\bar x) \setminus B_r(y(\bar x)))<3.5 \varepsilon$ holds. Clearly, for every $\varepsilon$ and $y(\bar x)$ there exists such point $w \in \partial B_\varepsilon(\bar x) \setminus B_r(y(\bar x))$, that $(w\bar x) \parallel (ix)$. Then in view of Lemma~\ref{Losiu} there exists such a connected set $S \ni w$, that $\H(S)<\varepsilon/4$ and $F_M(\Sigma \setminus B_\varepsilon(\bar x)\cup S) \leq F_M(\Sigma)$. Thus the set 
\[
\Sigma \setminus B_\varepsilon(\bar x) \cup (\partial B_\varepsilon(\bar x) \setminus B_r(y(\bar x))) \cup S,
\]
is strictly shorter than $\Sigma$ and has not larger energy. So we have a contradiction with an optimality of $\Sigma$.  

Let now $n=3$. Let us denote points $\Sigma \cap \partial B_\varepsilon(\bar x)$ by $A_1$, $A_2$ and $A_3$.  For each location of these points there exists two of them (without loss of generality let it be $A_1$ and $A_2$) form with $\bar x$ an isosceles triangle with an angle at the apex less than $\angle A_1 \bar x A_2 \leq \pi/2+o(\varepsilon)$. Then, in view of Proposition~\ref{angles4} the set $St(A_1, \bar x, A_2) \cup [A_3 \bar x]$ has the length not greater than $(3-c)\varepsilon$ where $c>0$ is a constant and $St(A_1, \bar x, A_2)$ is a Steiner tree for these three points. Thus the set
\[
\Sigma \setminus B_\varepsilon(\bar x) \cup (St(A_1, \bar x, A_2) \cup [A_3 \bar x] \cup S),
\]
where $S$ is constructed in Lemma~\ref{Losiu} for $w=\bar x$, is strictly shorter than $\Sigma$ and has not greater energy. The contradiction with the optimality of $\Sigma$ implies $n\leq 2$.
\end{proof}

\begin{corollary}
In view of~\ref{rem_ord} for arbitrary $x\in E_\Sigma$ there exists such $\rho >0$, that for every $\bar x \in G_\Sigma \cap B_\rho(x)\setminus \{x\}$the inequality $\oord_{\bar{x}}\Sigma \leq 2$ holds.
\end{corollary}
\begin{proposition}

For an arbitrary point $x\in E_\Sigma$ we will prove some facts concerning energetic points in its sufficiently small neighbourhood.
\begin{enumerate}
\item\label{W}
The following holds: $|\angle ((s^\varepsilon \bar x), (ix))-\pi/2|=o_\rho(1)$ if $s^\varepsilon \defeq \partial B_{\varepsilon}(\bar x) \cap \Sigma$ and $B_\varepsilon(\bar x) \subset (B_\rho(x)\setminus \{x\})$.

In other words, for each $c>0$ there exists such $\rho>0$ that for every $\bar{x} \in B_\rho(x) \setminus \{ x\}$ such that $\oordball_{\bar x}\Sigma=1$, for every $\varepsilon<\min(|\bar x x|, \rho-|\bar x x|)$ such that $\sharp(\partial B_{\varepsilon}(\bar x) \cap \Sigma)=1$ the inequality $|\angle ((s^\varepsilon \bar x), (ix))-\pi/2|<c$ holds, where $s^\varepsilon = \partial B_{\varepsilon}(\bar x) \cap \Sigma$.
\begin{proof}
Assume the contrary, id est there exists such $c>0$ that for every $\rho>0$ there exists such $\varepsilon_0>0$ that $\bar x \in \Sigma \cap B_{\rho-\varepsilon_0}(x) \setminus \overline{B_{\varepsilon_0} (x)} $ and the inequality $|\angle ((s_0\bar x), (ix))-\frac{\pi}{2}| > c$ holds for $s_0 \defeq \partial B_{\varepsilon_0}(\bar x) \cap \Sigma $. 

Let the point $w^0$ be the foot of the perpendicular from $s_0$ to the line, parallel $(ix)$ and containing $\bar x$:  $(w^0\bar x)\ ||\ (ix)$ and $(w^0s_0) \perp (ix)$. Then we denote the set $S \ni w^0$ as in Lemma~\ref{Losiu} for $w=w^0$. Then the length of this set is small: $\H(S)=o_\rho(1)\varepsilon_0$.

At the same time the inequality
\[
|w^0s_0|=\sin \angle((s_0\bar x],(ix]) \cdot |s_0\bar x|< \cos c |s_0\bar x| < (1 - \frac{c^2}{2}) \varepsilon_0
\]
holds for sufficiently small $c$. Herewith $\H(\Sigma \cap B_{\varepsilon_0}(\bar x)) \geq \varepsilon_0$. Thus \[
\H(\Sigma \setminus  B_{\varepsilon_0}(\bar x) \cup [s_0w^0] \cup S)\leq \H(\Sigma)-\varepsilon_0+\left(1-\frac{c^2}{2}\right)\varepsilon_0+o_\rho(1)\varepsilon_0<\H(\Sigma).
\]
Hence $\Sigma \setminus  B_{\varepsilon_0}(\bar x) \cup [s_lw^l] \cup S$ is a connected set with the length less than and energy not greater than $\Sigma$. The contradiction with optimality of $\Sigma$ completes the proof.
\end{proof}
\item\label{WW}
Let $x \in E_\Sigma$. Then for sufficiently small $\rho>0$ for each $\bar x \in G_\Sigma \cap B_\rho(\Sigma)\setminus \{ x\}$ with a circular order  $\oordball_{\bar x}\Sigma = 2$, there exists such a  small $\varepsilon>0$ that for every $s^l, t^l \defeq \partial B_{\varepsilon_l}(\bar x) \cap \Sigma$ where $\varepsilon_l \leq \varepsilon$ the inequalities $\angle s^l\bar xt^l > \frac{2\pi}{3} - \varepsilon_lo_\rho(1)$, $\angle ((s^l\bar x), (ix)) \leq \pi/3 +o_\rho(1)$ and $\angle ((t^l\bar x), (ix)) \leq \pi/3 +o_\rho(1)$ hold. 
\begin{proof}
The second and third inequalities follow from the first one and the fact that there exists such a point $y(\bar x)$ that $B_r(y(\bar x)) \cap \Sigma = \emptyset$ and $|\angle ((y(x)x),(ix)) - \pi/2| =o_\rho(1)$. The first inequality is true otherwise the connected set $\Sigma \setminus B_{\varepsilon_l}(\bar x) \cup St(s^l,t^l,\bar x) \cup S$, where $S \ni \bar x$ with length $\varepsilon_l o_\rho(1)$ is defined in Lemma~\ref{Losiu} for a point $w=\bar x$, and $St(s^l,t^l,\bar x)$ is a Steiner tree connecting points $s^l$, $t^l$ and $\bar x$ has not larger energy and strictly less length than $\Sigma$, which is impossible.
\end{proof}

\item\label{WWW}
Let $x \in E_\Sigma$. Then for sufficiently small $\rho>0$ for each $\bar x \in G_\Sigma \cap B_\rho(\Sigma)\setminus \{ x\}$ with a circular order  $\oordball_{\bar x}\Sigma = 2$, there exists such a  small $\varepsilon>0$, that for every $u \in \Sigma \cap B_\varepsilon(\bar x)$ the inequality $\angle((u \bar x),(ix)) \neq \pi/2$ holds.

In other words if the point $\bar x \in G_\Sigma \cap B_\rho(\Sigma)\setminus \{ x\}$ has a circular order $\oordball_{\bar x}\Sigma = 2$, there doesn't exist any consequence $\{u_m\}\subset \Sigma$ converging to $\bar x$ such that $\angle((u_m\bar x),(ix))=\pi/2$. 



\begin{proof} 


Assume the contrary to the statement of the item: that there exists such sequence $\{u_m\}$, converging to $\bar x$ that $\angle((u_m\bar x),(ix))=\pi/2$ for each $m$. 
Let $\{\varepsilon_l\}$ be such a decreasing to $0$ consequence that $\sharp (\partial B_{\varepsilon_l}(\bar x) \cap \Sigma )=2$ and $\varepsilon_1 < \min (|\bar x x|, \rho-|\bar x x|)$. Denote  $s^l, t^l \defeq \partial B_{\varepsilon_l}(\bar x) \cap \Sigma$. Then, from item~\ref{WW} the inequalities $\angle s^l\bar xt^l > \frac{2\pi}{3} - \varepsilon_lo_\rho(1)$, $\angle ((s^l\bar x), (ix)) \leq \pi/3 +o_\rho(1)$ and $\angle ((t^l\bar x), (ix)) \leq \pi/3 +o_\rho(1)$ hold.
Then 
\[
\sharp (\Sigma \cap \partial B_{|\bar x u_l|}(\bar x))\geq 3.
\]
for sufficiently small $\varepsilon_1>|u_1\bar x|$.

Note that for every $\varepsilon_l$ the following is true:
\begin{equation}
\H(B_{\varepsilon_l}(\bar x)\cap \Sigma)=2\varepsilon_l+o_\rho(1)\varepsilon_l.
\label{1000}
\end{equation}
Really $\H(B_{\varepsilon_l}(\bar x)\cap \Sigma) \geq 2\varepsilon_l$ in view of co-area inequality. And in other hand $\H(B_{\varepsilon_l}(\bar x)\cap \Sigma)\leq 2\varepsilon_l+o_\rho(1)\varepsilon_l$ because otherwise changing in the set $\Sigma$ the subset $\Sigma \cap B_{\varepsilon_l}(x)$ by the set $ [\bar xs^l]\cup [\bar x t^l]\cup S$, where $s_l,t_l=\partial B_{\varepsilon_l}(\bar{x})\cap \Sigma$, the $S$ is defined at the Lemma~\ref{Losiu} for $w=\bar x$ and has a length $\varepsilon_l o_\rho(1)$, saves the connectivity, reduces the length and not increases the energy, which is impossible.
So one can choose such small $\varepsilon_1$ that
\begin{equation}
\H(B_{\varepsilon_l}(\bar x)\cap \Sigma)<2\varepsilon_l \left (1+\frac{1}{36}\right).
\label{1010}
\end{equation}
Note then that there exists such sequence (add elements to the sequence if necessary) $\{\varepsilon_l\}$ that
\begin{equation}
\frac{\varepsilon_l}{\varepsilon_{l+1}}\leq \frac{18}{17}.
\label{1030}
\end{equation}
Really, assume the contrary, then there exists such $\varepsilon_l$ that $\sharp \left( \partial B_{\varepsilon_l}(\bar x)\cap \Sigma
\right) \geq 3$ for every $\varepsilon \in [\frac{17\varepsilon_l}{18},\varepsilon_l[$. But then, in view of the co-area inequality,
\[
\H(\Sigma \cap B_{\varepsilon_l}(\bar x))\geq \H(\Sigma \cap B_{\frac{17\varepsilon_l}{18}}(\bar x)) +\H(\Sigma \cap (B_{\varepsilon_l}(\bar x) \setminus  B_{\frac{17\varepsilon_l}{18}}(\bar x))\geq \varepsilon_l\left(1+\frac{1}{18}\right),
\]
which contradicts to \eqref{1010}.
Consider an arbitrary $u_m$ and choose such a number $l$ that 
\begin{equation}
0<\varepsilon_l-|\bar xu_m|\leq \frac{1}{18}\varepsilon_l.
\label{1040}
\end{equation}
Such $\varepsilon_l$ exists in view of the inequality~\eqref{1030}. Also we for every $\delta \in ]|\bar x u_m|, \varepsilon_l[$ the inequality $\sharp \left( \partial B_{\delta}(\bar x)\cap \Sigma
\right) \geq 3$ holds (one can add new element to the sequence if necessary).
Let \[\gamma \defeq \max_{s \in \Sigma \cap \overline B_{\varepsilon_l}(\bar x): \sharp (B_{|\bar x s|}(\bar x) \cap \Sigma) <3)} \angle{((ix),(\bar x s)}.\] Then $\gamma\leq \frac{\pi}{3}+ o_\rho(1)$.
Note that~\eqref{1010} implies that
\[
\H(\Sigma \cap \Lambda((i \bar{x}),\gamma) \geq 2\varepsilon_l(1-\frac{1}{36})
\]
Let $q$ be a connected component of $\Sigma \setminus \Lambda((i \bar{x}),\gamma)$ containing $u_m$ (by point $i$ we mean sufficiently far point from $(ix]$).
Then 
\[
\H(q) \geq \dist (u_m, \Lambda((i\bar x),\gamma)) 
=\cos \gamma |u_m \bar x|> \frac{|u_m \bar x|}{3}.
\]
Thus
\[
\H(\Sigma \cap B_{\varepsilon_l}(\bar x))\geq \H(\Sigma \cap B_{\varepsilon_l}(\bar x)\cap \Lambda((i\bar x),\gamma))+\H(q)\geq 2\varepsilon_l(1-\frac{1}{36})+\frac{|u_m \bar x|}{3}>2\varepsilon_l(1+\frac{1}{36}),
\]
which contradicts to~\eqref{1010}. The last inequality follows from~\eqref{1040}.
\end{proof}
\item\label{WWWW}
For every $x\in E_\Sigma$  for each $c>0$ there exists such $\rho>0$, that for every $\bar{x} \in B_\rho(x) \setminus \{ x\}$ for every $\varepsilon$ such that $B_\varepsilon(\bar x)\subset B_\rho(x) \setminus \{x\}$ for $S,T \defeq \partial B_\varepsilon(\bar x)$ the inequality $|\angle ((S \bar x), (ix)) - \angle ((T \bar x), (ix))|<c$ holds.
In other words, $|\angle ((S \bar x), (ix)) - \angle ((T \bar x), (ix))| = o_\rho(1)$

\begin{proof}
Assume the contrary: let there exist such $c>0$, that
\[
\eta \defeq \angle((S\bar x), (ix))-\angle( (ix),(T\bar x)) \geq c
\]
for infinitesimal $\rho$.
Let $l$ be the line containing $\bar x$ and parallel to $(ix)$. Let point $T'$ be a reflection of $T$ relative to $l$. Then $|T'\bar x|=\varepsilon$ and
\[
|S \bar x|+|T \bar x|=|S\bar x|+|T' \bar x|=2\varepsilon.
\]
At the same time in view of law of cosines for triangle $\Delta S\bar xT'$
\[
|ST'|^2 = |S\bar x|^2+|\bar xT'|^2-2|S\bar x||\bar x T'|\cos \angle S\bar xT'=2\varepsilon^2(1-\cos \eta),
\]
as
\[
\angle S\bar xT'=\pi+\angle((T\bar x), (ix))-\angle( (ix),(S\bar x)).\]
Let $d \defeq [ST']\cap l$. Then $\angle((Sd], (ix))=\angle( (ix),(T'd])$ and 
\[(|S\bar x|+|T\bar x|)-(|Sd|+|dT|)= 2\varepsilon - |ST'|=\varepsilon(2-\sqrt 2+\sqrt 2 \cos \eta).\]

In view of Lemma~\ref{Losiu} for $w=d$ there exists a connected set $S \ni d$ such that  $\H(S)<(2-\sqrt 2+\sqrt 2 \cos \eta)\varepsilon$ and $F_M(\Sigma \setminus B_\varepsilon(\bar x) \cup S)\leq F_M(\Sigma)$. Then the connected set $\Sigma \setminus ([S\bar x]\cup[T\bar x]) \cup S \cup [sd] \cup [td]$ is better than $\Sigma$, as is energy is not greater and its length is strictly less. 
So we have a contradiction with optimality of $\Sigma$.
\end{proof}

\end{enumerate}
\label{items}
\end{proposition}

Now we can prove the following lemma.
\begin{lemma}
For every $x\in E_\Sigma$ and  $c>0$ there exists such $\rho>0$, that for every $\bar x \in G_\Sigma \cap B_\rho(x)  \setminus \{x\}$ and for the $(ix]$, closest to $\bar x$ energetic ray of $x$, the following statements are true: 
\begin{enumerate}
\item if $\oord_{\bar x}\Sigma=1$, then $\bar x \in X_\Sigma$ and for the segment $[s\bar{x}] \subset \Sigma$ the following inequality is true:  \[\angle((s\bar{x})(ix))>\pi/2-c.\] 
In other words the segment $[s\bar x]$ is $o_\rho(1)-$almost perpendicular to $(ix)$.
If $(ix]$ is one-oriented energetic ray, then for the point $y\defeq y(x)$, which is the limit of the points corresponding to the energetic points of the energetic ray $(ix]$ the following inequality is true:
 \[
 \angle ([s\bar{x}), [xy))<c;
 \]
\label{__1}
\item if $\bar x \in X_\Sigma$ and  $\oord_{\bar x} \Sigma=2$, then the angle of incidence is almost equal to the angle of reflection:
\[|\angle((s\bar{x}], (ix))-\angle( (ix),(t\bar{x}])|<c\mbox{, where } [s\bar{x}] \cup [t\bar{x}] \subset \Sigma;\]
\label{__2}
\item if $\bar {x} \in E_\Sigma$, then there exists such $\varepsilon>0$ that $B_\varepsilon(\bar x) \cap \Sigma$ is contained in the union of two cusps (and each cusp is not empty) with axes, almost parallel to $(ix)$, id est $|\angle((m\bar{x}],(ix))|<c$ and $|\angle((n\bar{x}],(ix))|<c$, where $\Sigma \cap B_\varepsilon(\bar x) \subset Q_\varepsilon^{m,\bar x}\cup Q^{n, \bar{x}}_\varepsilon$, for $ Q_\varepsilon^{m,\bar x} \in  \mathbb{Q}_\varepsilon^{m,\bar x}$, $ Q_\varepsilon^{n,\bar x} \in \mathbb{Q}_\varepsilon^{n,\bar x}$. 
\label{__3}
\end{enumerate}
\label{alm}
\end{lemma}
\begin{proof}
\begin{enumerate}

\item Lets prove that  
 $\oord_{\bar x}\Sigma = 1$ implies $\bar x\in X_\Sigma$. 
\label{QQ}
Assume the contrary: let $\bar x \in E_\Sigma \cap B_\rho(x) \setminus x$. 

In view of Lemma~\ref{ordball_} it is suffices to analyze the following cases:
\begin{itemize}
\item [(a)] Case $\oordball_{\bar x}\Sigma=2$.
Let $\{\varepsilon_l\}$ be a decreasing to $0$ sequence of such numbers that $\sharp (\partial B_{\varepsilon_l}(\bar x) \cap \Sigma )=2$ and $\varepsilon_1 < \min (\rho, \rho-|\bar x x|)$. Denote  $s^l, t^l \defeq \partial B_{\varepsilon_l}(\bar x) \cap \Sigma$. 
We know from the item~\ref{WW} of Proposition~\ref{items} the inequalities $\angle ((s^l\bar x), (ix)) \leq \pi/3 +o_\rho(1)$ and $\angle ((t^l\bar x), (ix)) \leq \pi/3 +o_\rho(1)$ hold.
Then, due to the unit order of $\bar x$ there exists a converging to $\bar x$ sequence of such points $\{u_m\}\subset \Sigma$ that \[\angle\left((u_l\bar x),(ix)\right)=\frac{\pi}{2}.\] So we have a contradiction to the item~\ref{WWW} of Proposition~\ref{items}.

\item [(b)] Case $\oordball_{\bar x}\Sigma=1$. 

Let $\varepsilon_l>0$ be a decreasing to $0$ sequence of such numbers that $\sharp (\partial B_{\varepsilon_l}(\bar x) \cap \Sigma )=1$ and $s_l\defeq \partial B_{\varepsilon_l}(\bar x) \cap \Sigma$. Then in view of item~\ref{W} of Proposition~\ref{items} each line $(s_l\bar x)$ is almost perpendicular to $(ix)$:
\begin{equation}
\left|\angle \left((s_l\bar x), (ix)\right)-\frac{\pi}{2}\right| =o_\rho(1).
\label{99}
\end{equation}
Note that for every $\varepsilon_l$ holds
\begin{equation}
\H(B_{\varepsilon_l}(\bar x)\cap \Sigma)=\varepsilon_l+o_\rho(1)\varepsilon.
\label{100}
\end{equation}
The inequality $\H(B_{\varepsilon_l}(\bar x)\cap \Sigma)\geq \varepsilon_l$ is true because of connectivity $\Sigma$ and the inequality in the other side is true in view of Lemma~\ref{Losiu}: otherwise one could replace $\Sigma$ by $\Sigma \setminus B_{\varepsilon_l}(\bar x) \cup [\bar{x}s_l]\cup S$, where $S$ is defined at Lemma~\ref{Losiu} for $w=\bar x$, and get a contradiction with optimality of $\Sigma$. 

So, for sufficiently small $\rho$  \begin{equation}
\H(B_{\varepsilon_l}(\bar x)\cap \Sigma)<\varepsilon_l \left(1+\frac{1}{18}\right).
\label{101}
\end{equation}
Then we can say (add elements to $\{\varepsilon_l\}$ if it is necessarily) that
\begin{equation}
\frac{\varepsilon_l}{\varepsilon_{l+1}}\leq \frac{18}{17}.
\label{103}
\end{equation}
Otherwise there exists such $\varepsilon_l$ that $\sharp ( \partial B_{\bar{\varepsilon}}(\bar x)\cap \Sigma)\geq 2$ for every $\bar{\varepsilon} \in [\frac{17\varepsilon_l}{18},\varepsilon[$. Thus, by co-area inequality, 
\[
\H \left(\Sigma \cap B_{\varepsilon_l}(\bar x)\right)\geq \H \left(\Sigma \cap B_{\frac{17\varepsilon_l}{18}}(\bar x)\right) +\H \left(\Sigma \cap B_{\varepsilon_l}(\bar x) \setminus  B_{\frac{17\varepsilon_l}{18}}(\bar x)\right)\geq \varepsilon_l \left (1+\frac{1}{18} \right),
\]
which contradicts to~\eqref{101}.
Note, as $\bar x$ has an energetic ray which is, in view of Corollary~\ref{alm_parr} $o_\rho(1)$-almost parallel to $(ix)$, there exists a sequence of energetic points $\{x_l\}$, converging to $\bar x$ such that $\angle ((x_l\bar x], (ix]) =o_\rho(1)+o_{|x_l\bar x|}(1)$.  The sequences $\{x_l\}$ and $\{ s_l \}$ with unit order $\oord_{\bar x}(\Sigma)=1$ imply the presence of also converging to $\bar x$ sequence of points $\{ u_l\}\subset \Sigma$ such that $\angle \left((u_l\bar x),(ix)\right) = \frac{\pi}{4}$. Consider an arbitrary element $u_m$ of this sequence. In view of the inequality~\eqref{103} there exists such number $l$ that
\begin{equation}
0<\varepsilon_l-|\bar xu_m|\leq \frac{1}{18}\varepsilon_l.
\label{104}
\end{equation}
Note also, that by~\eqref{99} the inequality
\[
\sharp (\partial B_{|\bar x u_m|}(\bar x)\cap \Sigma) \geq 2
\] 
holds.


Let 
\[
k:=\max_{s \in \Sigma \cap \overline{B_{\varepsilon_l}}: \sharp (\Sigma \cap _{|s \bar x|}(\bar x))=1} \angle((ix),(\bar xs)),
\]
then $|k-\pi/2|=o_\rho(1)$.
Define also $K := \Lambda((p \bar x),k-\pi/2)$, where $(p\bar x) \perp (ix)$. $K$ contains an area of possible location of points $s_l$.

Let $Q$ be a cusp or a union of two cusps with vertex at $\bar x$, containing all energetic points of  $B_{\varepsilon_l}(\bar x)$, then
\begin{equation}
d \defeq \dist(u_m, Q \cup K)\geq \frac{\sqrt{2}}{2}|\bar xu_m|-o_\rho(1)\varepsilon_l \geq (\frac{17}{18\sqrt 2}+o_\rho(1))\varepsilon_l,
\end{equation}
where the last inequality follows from~\eqref{104}. 
Then the length of the connected component of $\Sigma \setminus (Q \cup K)$, containing $u_m$, is greater or equal to $d$.
In view of co-area inequality, item~\ref{W} of Proposition~\ref{items} and inequality ~\eqref{101} the inequality $\H(\Sigma \cap K \cap B_\varepsilon(\bar x)) \geq \frac{17}{18}\varepsilon$ holds.
On the other hand \[
\H(\Sigma \cap B_{\varepsilon_l}(\bar x)) \geq H(\Sigma \cap B_{\varepsilon_l}(\bar x)\cap K)+\H(\Sigma \cap B_{\varepsilon_l}(\bar x) \setminus (K \cup Q)) \geq \frac{17}{18}\varepsilon_l +d\geq
\varepsilon_l (\frac{17}{18}+\frac{17}{18\sqrt 2}-o_\rho(1))>\frac{19}{18}\varepsilon_l.
\]
So we get a contradiction to the inequality~\eqref{101}.
\end{itemize}
\item
The statement immediately follows from item~\ref{WWWW} of Proposition~\ref{items}.

\item
Because of items~\ref{BB} and~\ref{reb} Lemma~\ref{con} there exists $Q$ --- a cusp or a union of two cusps with vertex at $\bar x$ such that $G_\Sigma \cap B_{\varepsilon_0}(\bar x) \subset Q$ for sufficiently small $\varepsilon_0>0$. In view of Corollary~\ref{alm_parr} axes of these cusps are almost parallel to $(ix)$. If at least one energetic ray of $\bar x$ is two-orientated, the statement of the item is trivial because of Remark~\ref{star}. Similarly, the item is trivial if the rays are orientated to the different sides, id est for sufficiently small $\varepsilon_0>0$ points of $M$, corresponding to the energetic points from the set $B_{\varepsilon_0}(\bar x) \setminus \bar{x}$ are locates at both parts of the plane divided by energetic rays of $\bar x$. 
Let $(q_l\bar x]$ be a ray, not coinciding with any of cusps' axes. Then, because of the definition of cusps, there exists such $\varepsilon_1>0$, that $(q_1\bar x] \cap Q \cap B_{\varepsilon_1}(\bar x)=\{\bar x\}$. We can also assume $\sharp (\Sigma \cap \partial B_{\varepsilon_1}(\bar x))=2$, as  $\oord_{\bar x} \Sigma =\oordball_{\bar x}\Sigma=2$ because of the item~\ref{QQ} of the proven Lemma, Remark~\ref{rem_ord} and Lemma~\ref{ordball_}. We will show that there exists such $\varepsilon_2>0$, that 
\[
B_{\bar \varepsilon_2}(\bar x) \cap (q\bar x] \cap \Sigma = B_{\bar \varepsilon_2}(\bar x) \cap (q\bar x], 
\]
or 
\[
B_{\bar \varepsilon_2}(\bar x) \cap (q\bar x] \cap \Sigma = \bar x. 
\]

 Assume neither of the two cases come true. Then the set $\Sigma \cap B_{\varepsilon_1}(\bar x) \cap (q\bar x]$ contains infinitely many points close to $\bar x$. Herewith each point of the set $\Sigma \cap B_{\varepsilon_1}(\bar x) \cap (q\bar x]$ is not energetic and should be a center of a segment or a regular tripod.  
By item~\ref{WWW} Proposition~\ref{items} ray $(qx]$ can not be perpendicular to $(ix]$. Let $L_1$ be an open area bounded by perpendicular to $(ix]$, passing through the point $\bar x$, $(q_1\bar x]$ and $\partial B_{\varepsilon_1}(\bar x)$, not containing any energetic points. Then $S := \overline{L_1\cap \Sigma}$ is a Steiner forest, with all infinitely many terminal points, except at most three from the set $\partial B_{\varepsilon_1} (\bar x) \cup \{ \bar x \}$, located at the perpendicular or at $(q\bar x]$. By Lemma~\ref{St_l} the perpendicular contains infinite number of terminal points of $ S$, thus there exists an infinite sequence of points $\{u_l\}$ such that $\angle((u_l,\bar x)(ix))=\pi/2$ for each $l$ which contradicts to item~\ref{WWW} of Proposition~\ref{items}. Ergo there exists such $\varepsilon_1>0$, that $\Sigma \cap (q_1\bar x] \cap \overline {B_{\varepsilon_1}(\bar x)}$ coincides with $\{\bar x\}$ or $(q_1\bar x] \cap B_{\varepsilon_1}(\bar x) $. We pick such $\varepsilon_1$ that also $\sharp(\partial B_{\varepsilon_1}(\bar x)\cap \Sigma)=2$ and $Q \cap B_{\varepsilon_1}(\bar x) \cap (q_1\bar x] = \emptyset$.

Let $(m\bar x]$ be an energetic ray of $\bar x$. Pick such decreasing to zero sequences $\{\bar{\varepsilon}_i \}$ and $\{ \gamma_i\}$, that 
\begin{itemize}
\item $Q \cap B_{\bar \varepsilon_i}(\bar x)$ is contained in the angle with bisector $(m\bar x]$ and size $2\gamma_i$;
\item the intersection of the borders of this angle (with bisector $(m\bar x]$ and size $2\gamma_i$) with the set $\Sigma \cap \cap B_{\bar\varepsilon_{i+1}}(\bar x)$ coincides with the point $\{ \bar x \}$ or with a boundary of a sector of $B_{\bar \varepsilon_{i+1}}(\bar x)$. 
\end{itemize}
Let $\gamma_1 \leq \pi/12$, $\bar \varepsilon_1 \defeq \varepsilon_1$.
If at some number (at most at one by Lemma~\ref{angles2}) the intersection is the segment, let $j$ be equal to this number. Otherwise let $j=0$.

Define the following cusp: \[
Q_1 \defeq Q^{m,\bar x}_{\overline{\varepsilon}_{j+1}}\defeq \bigcup_{i=j+1}^{\infty} \left(Q(m,\overline x, \gamma_i, \overline{\varepsilon_i}) \setminus B_{\overline{\varepsilon}_{i+1}}(\bar x)\right).
\]
If $\bar x$ has two energetic rays, similarly define the second cusp $Q_2$ with second energetic ray as the axis.

Thus an arbitrary connected component of $\Sigma \cap B_{\bar{\overline{\varepsilon_1}}}(\bar x)\setminus \{\bar x\}$ located inside the cusp $Q_1$ and $Q_2$ if it exists, or located outside of it. As $\bar x$ has an energetic ray, at least on of the components should lay inside, then it has a tangent line at the point $\bar x$, which coincides to energetic ray and so $o_\rho(1)-$almost parallel to $(ix]$. Then by Lemma~\ref{angles2} another component can not lay in the same cusp. 
If there is only one cusp there could be one ray $(q \bar x]$, which doesn't intersect the cusp and $\varepsilon_1>0$ such that a connected component in $B_{\varepsilon_1}(\bar x)$ coincides with the segment $(qx]\cap B_{\varepsilon_1}(\bar x)$. 
Then we can choose an arbitrary cusp with $(qx]$ as an axis. Otherwise similarly construct the second cusp for the second energetic ray with decreasing sequence of radii $\{\varepsilon_k\}$.
Note that as $\bar x$ (due to item~1 of this lemma) can not have a unit order due to the item~3 from the proof of the Lemma~\ref{angles2}) the set $\Sigma \cap B_{\bar \varepsilon}(\bar x)$ can not be located in the one cusp for any $\bar \varepsilon>0$. So by Lemma~\ref{angles2} there is certainly two cusps. Applying an item~\ref{WWWW} Proposition~\ref{items} we get that there is an infinite strictly decreasing to $0$ sequence $\varepsilon_k$ such that for $S_k,T_k\defeq \partial B_{\varepsilon_k}(\bar x)$ the inequality $|(\angle S \bar x], (ix] - \angle (T \bar x], (ix]| = o_\rho(1)$ holds, so the axes of both cusps should be $o_\rho(1)$-almost parallel to $(ix]$, which is required for $\varepsilon$ equal to minimum of $\varepsilon_1$ and $\overline{\varepsilon_1}$.  
\end{enumerate}
 
 \end{proof}

\begin{corollary}
For every $x\in E_\Sigma$ and $c>0$ there exists such $\rho>0$ that for each point $\bar x \in B_\rho(x)\cap \Sigma \setminus \{ x\}$ there exists a one-sided tangent of $\Sigma$ at it. 
Moreover if $\bar x \in E_\Sigma$ then there are exactly two one-sided tangent at the $\bar x$ with angle, greater than $\pi-c$ between it.
Existence of one-sided tangents (one, two or three of them) and the size of angles between them in case $\bar x \in X_\Sigma \cup E_\Sigma$ immediately follows from Proposition~\ref{discr_str}. 
\label{1sdd}
\end{corollary}
\begin{remark}
By the compactness of $E_\Sigma$ one can swap the predicates in Lemma~\ref{alm}: $ \forall c>0 \ \exists \rho>0 \ \forall x \in E_\Sigma \ldots$
\end{remark}

\begin{corollary}

In the notations of item~\ref{__2} Lemma~\ref{alm} 
\begin{enumerate}
    \item the inequality $|\angle t \bar x y - \angle y \bar x s| <c$ holds, where $y$ is the limit point of the points corresponding to energetic points of the energetic ray $(ix]$ converging to $x$; 
    \item also the inequality  $|\angle s\bar xt| \geq 2\pi/3$ holds, and hence $\angle((s\bar{x}), (ix)) \leq \pi/6+c$ and $\angle((t\bar{x}), (ix)) \leq \pi/6+c$;
    \item 
let there exists a path in $\Sigma$ from $\bar x$ to $x$ containing $t$. Then $\angle((s\bar{x}], (ix]) \leq \pi/6+c$ and $\angle((t\bar{x}], (ix]) \geq 5\pi/6-c$. Moreover, if $\angle ixy >0$ (i.e. equal to $\pi/2$ up to $c$), then $\angle s\bar x y>0$ (and belongs to $[\pi/2-c; \frac{2}{3}\pi+c]$) and $\angle t\bar x y<0$.
\end{enumerate}
\label{sm_an}
\end{corollary}
\begin{proof}
The first item is trivial and the second one follows from the first by Lemma~\ref{angles2}. So we only have to prove the third one. Let $\Sigma'$ be a connected component of $\Sigma \setminus \{ x\}$. Then there exists such open area $U \subset B_\rho(x)$ that $\sharp (\partial U \cap \Sigma')=1$. We can also assume that $\bar x, s, t \in U$. 
Note that there exists such $\bar y=y(\bar x) \in M$ that 
\[
B_r(\bar y) \cap \Sigma =\emptyset = B_r(y) \cap \Sigma  
\]
and 
\[
B_r(\bar y) \cap B_r(y) \neq \emptyset.
\]
Thus, if the proving statement is not true and $\angle((s\bar{x}], (ix]) \geq 5\pi/6-c$ then the point $s$ is inside the area, bounded by $\partial B_r(y)$, $\partial B_r(\bar y)$ and by the path from $\bar x$ to $x$, located in $B_\rho(x)$. But if there are any energetic points in this area, the points of $M$, corresponding to it, are located in the same half-plane with respect to $(ix)$ as $y$. But the path, that began from $\bar x-t$ should contain an energetic point with corresponding point from the another half-plane. 
\end{proof}


Let us prove the following technical lemma:
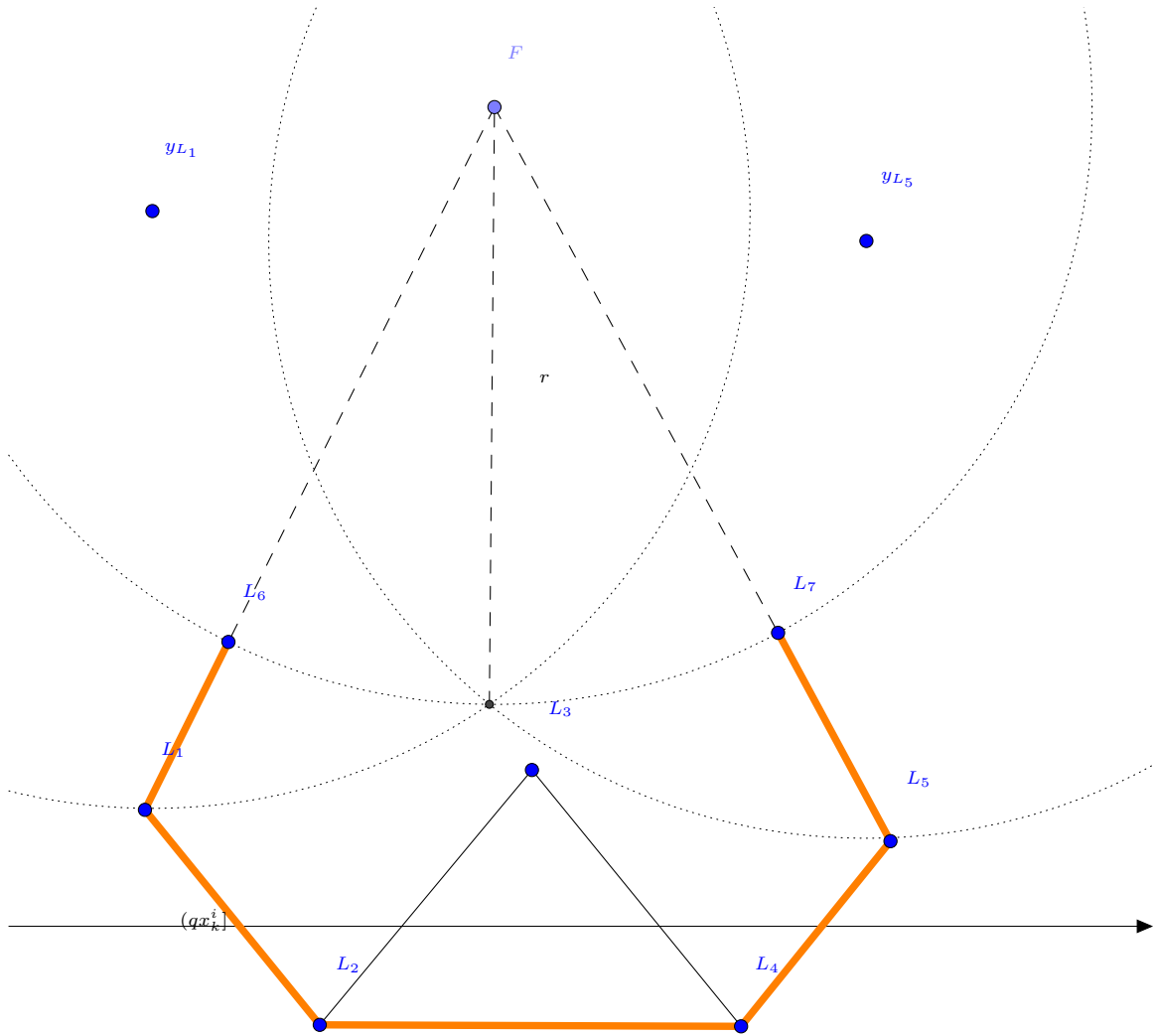
\begin{figure}[h]
\begin{center}
\definecolor{ffxfqq}{rgb}{1.,0.4980392156862745,0.}
\definecolor{xdxdff}{rgb}{0.49019607843137253,0.49019607843137253,1.}
\definecolor{uuuuuu}{rgb}{0.26666666666666666,0.26666666666666666,0.26666666666666666}
\definecolor{qqqqff}{rgb}{0.,0.,1.}
\begin{tikzpicture}[line cap=round,line join=round,>=triangle 45,x=1.0cm,y=1.0cm]
\draw[->,color=black] (-6.041676947400006,0.) -- (9.275239458600005,0.);
e
\clip(-6.041676947400006,-2.387331637299997) rectangle (9.275239458600005,12.305995296699988);
\draw (-4.22,1.56)-- (-1.88,-1.32);
\draw (-1.88,-1.32)-- (0.9604018000000019,2.0934775999999964);
\draw (0.9604018000000019,2.0934775999999964)-- (3.76,-1.34);
\draw (5.76,1.14)-- (3.76,-1.34);
\draw [dotted](-4.12,9.58) circle (8.cm);
\draw[dotted](5.437619600000005,9.17972159999999) circle (8.cm);
\draw [dash pattern=on 5pt off 5pt](0.3904769433497686,2.972761715852066)-- (0.4586741801717944,10.972471031626599);
\draw [dash pattern=on 5pt off 5pt] (0.4586741801717944,10.972471031626599)-- (-4.22,1.56);
\draw [dash pattern=on 5pt off 5pt] (0.4586741801717944,10.972471031626599)-- (5.76,1.14);
\draw [dotted](0.4586741801717944,10.972471031626599) circle (8.cm);
\draw [line width=2.8pt,color=ffxfqq] (-3.102243299513783,3.808682459727668)-- (-4.22,1.56)-- (-1.88,-1.32)-- (3.76,-1.34)-- (5.76,1.14)-- (4.255313224392387,3.930771523134365);
\begin{scriptsize}
\draw [fill=qqqqff] (-4.22,1.56) circle (2.5pt);
\draw[color=qqqqff] (-3.839626624400004,2.3675380866999984) node {$L_1$};
\draw [fill=qqqqff] (-1.88,-1.32) circle (2.5pt);
\draw[color=qqqqff] (-1.5011661044000029,-0.5165632212999987) node {$L_2$};
\draw[color=black] (-3.430396033400004,0.08753907970000074) node {$(qx^i_k]$};
\draw [fill=qqqqff] (0.9604018000000019,2.0934775999999964) circle (2.5pt);
\draw[color=qqqqff] (1.3439608615999994,2.913178874699998) node {$L_3$};
\draw [fill=qqqqff] (3.76,-1.34) circle (2.5pt);
\draw[color=qqqqff] (4.111139143600002,-0.5165632212999987) node {$L_4$};
\draw [fill=qqqqff] (5.76,1.14) circle (2.5pt);
\draw[color=qqqqff] (6.137804927600003,1.9777946666999988) node {$L_5$};
\draw [fill=qqqqff] (5.437619600000005,9.17972159999999) circle (2.5pt);
\draw[color=qqqqff] (5.864984533600003,10.006509118699991) node {$y_{L_5}$};
\draw [fill=qqqqff] (-4.12,9.58) circle (2.5pt);
\draw[color=qqqqff] (-3.7227035984000043,10.39625253869999) node {$y_{L_1}$};
\draw [fill=uuuuuu] (0.9271426566502357,15.786959884147928) circle (1.5pt);
\draw[color=uuuuuu] (-5.924753921400006,12.637277203699988) node {$D$};
\draw [fill=uuuuuu] (0.3904769433497686,2.972761715852066) circle (1.5pt);
\draw [fill=xdxdff] (0.4586741801717944,10.972471031626599) circle (2.5pt);
\draw[color=xdxdff] (0.739858560599999,11.701892995699989) node {$F$};
\draw[color=black] (1.1296019805999993,7.336766691699993) node {$r$};
\draw [fill=qqqqff] (-3.102243299513783,3.808682459727668) circle (2.5pt);
\draw[color=qqqqff] (-2.7483450484000036,4.4721525546999965) node {$L_6$};
\draw  [fill=qqqqff](4.255313224392387,3.930771523134365) circle (2.5pt);
\draw[color=qqqqff] (4.617805589600001,4.589075580699997) node {$L_7$};
\end{scriptsize}
\end{tikzpicture}
    \caption{The absence of five energetic points of degree two in a row (to statement~\ref{brrrr}).}
    \label{fig:dvapo}
    \end{center}
\end{figure}

\begin{proposition}
For an arbitrary $x \in E_\Sigma$ and $\bar c>0$ there exists such $\rho>0$, that if $\Sigma \cap B_\rho(x)$ contains a polyline $L_1-L_2-L_3-L_4-L_5$, the nodes of which are energetic points of the same energetic ray $(ix]$ and all its internal points are of order $2$, then $\angle((L_jL_{j+1}),(ix)) \leq \bar c$ for every $j=1, \dots, 4$ (see Fig.~\ref{fig:dvapo}).  
\label{brrrr}
\end{proposition}
\begin{proof}
Let $\rho>0$ be such small that conditions of Lemma~\ref{alm} for $c=\bar c/10$ holds. Assume the contrary to the desired statement: id est $\angle((L_jL_{j+1}),(ix) > \bar c$ for some $j\in\{1,\dots,4\}$. Due to  Lemma~\ref{alm} points $L_1$,\ldots, $L_5$ belong to $X_\Sigma$ and the inequality $\angle((L_jL_{j+1}),(ix))> \bar{c}/2$ holds for every $j=1,2,3,4$. By Remark~\ref{rem_cc}, for sufficiently small $\rho>0$ $\angle ((L_{j-1}L_{j}],(L_{j+1},L_j]) \leq \pi-\bar c/2$ for $j \in \{2,3,4\}$, so points of $M$ corresponding to $L_j$ with odd and even indices located at the different semi-plane with respect to $(ix)$. 
Note that the ratio between the lengths of the sides of the polyline is linear. Really, by law of sines and the inequality $\angle L_{j-1}L_jL_{j+1}\geq 2\pi/3$ we get
\[
\frac{|L_{j-1}L_j|}{|L_jL_{j+1}|}=\frac{\sin \angle L_{j-1}L_{j+1}L_j}{\sin \angle L_jL_{j-1}L_{j+1}}\geq \frac{\sin (c/2+o_\rho(1))}{\sin (\pi/3)}\geq \frac{\bar c}{3}.
\]
On the other hand
\[
\frac{|L_{j-1}L_j|}{|L_jL_{j+1}|}=\frac{\sin \angle L_{j-1}L_{j+1}L_j}{\sin \angle L_jL_{j-1}L_{j+1}}\leq \frac{\sin \frac{\pi}{3}}{\sin (\bar c/2+o_\rho(1))}\leq \frac{3}{\bar c}.
\]
Thus we get \[\frac{\bar c}{3} \leq \frac{|L_{j-1}L_j|}{|L_jL_{j+1}|}\leq \frac{3}{\bar c}\]
for every $j\in \{2,3,4\}$.
 Then the replacement the polyline $L_1-L_2-L_3-L_4-L_5$ by the polyline $L_1-L_2-L_4-L_5$ reduces the length by amount, linearly dependent on  $|L_1L_2|$:
 \begin{multline*}
 |L_2L_3|+|L_3L_4|-|L_2L_4|=\frac{(|L_2L_3|+|L_3L_4|)^2-|L_2L_4|^2}{|L_2L_3|+|L_3L_4|+|L_2L_4|}>\\
 > \frac{2|L_2L_3||L_3L_4| (1+ \cos \angle L_2L_3L_4)}{2(|L_2L_3|+|L_3L_4|)}\geq
 \frac{|L_3L_4|(1- \cos \frac{\bar{c}}{2})}{ \left (1+\frac{|L_3L_4|}{|L_2L_3|} \right)}\geq  \frac{\bar{c}^4}{24(1+\frac{3}{\bar{c}})}|L_1L_2|.
\end{multline*}
Let $A'$ be a figure bounded by $[L_1,L_5]$, $\partial B_r(y(L_1))$ and $\partial B_r(y(L_3))$. The boundary of $A'$ is the union $[L_1L_5] \cup \breve{L_1C} \cup \breve{L_5C}$ where $C$ is a point of intersection $\partial B_r(y(L_1))$ and $\partial B_r(y(L_3))$. Note that $\dist(C,(L_1L_5))=o(|L_1L_5|)=o(|L_1L_2|)=o(\rho)$. Thus there exists a set $S=S_1\cup S_5$ such that $\H(S) = o(|L_1L_2|)$, $L_j\in S_j$ and $S_j$ is connected for $j\in \{1;5\}$. And we also will impose some condition on the energy of $S$. Let $L_j'$ be such point that $L_1'$,$C$ and $L_5'$ lie on the same line parallel to $(L_1L_5)$ and $(L_jL_j')\perp (L_1L_5)$ for $j\in \{1;5\}$. And let $C'$ be such point that $|CC'|=r$,$(CC')\perp (L_1L_5)$ and $[CC')\cap (L_1L_5)=\emptyset$. Now let $L_j''\defeq \partial B_r(C')\cap [L_i'C']$ for $i\in \{1;5\}$. Then $S\defeq [L_1L_1']\cup [L_1'L_1'']\cup [L_5L_5']\cup [L_5'L_5'']$. Thus $\H(S)=o(|L_1L_2|)=o(\rho)$. 
It is easy to see that the connected set  
\[
S \cup [L_1L_2] \cup [L_2L_4]\cup [L_4L_5]
\]
for sufficiently small $\rho>0$ has strictly smaller length and not greater energy than $\Sigma$, so we get a contradiction with optimality of $\Sigma$.
  \end{proof}

\begin{lemma}
Let $x \in E_\Sigma$ and let there exist such a  cusp $Q\in \mathbb{Q}^{a,x}_\varepsilon$ that $Q^{a,x}_\varepsilon \cap \Sigma$ is an arc. Then for each $c>0$ there exists such $\varepsilon=\varepsilon(c)>0$ that every one-sided tangent at the points of the set $Q^{a,x}_\varepsilon \cap B_{\varepsilon(c)}(x) \cap \Sigma$ has an angle not greater than $2c$ with a line $(ax)$.
\label{zapred}
\end{lemma}

\begin{proof}
Assume the contrary, id est there exists a sequence of points with one-sided tangents having an angle with $(ax)$ greater than $2c$, converging to $x$. Without loss of generality assume $c<\pi/2$. By item~\ref{__3} of Lemma~\ref{alm} for sufficiently small $\varepsilon$ this sequence consists only on the points of $X_\Sigma \sqcup S_\Sigma$. As there is no branching points, the direction of tangents doesn't change at points of $S_\Sigma$, that is there is a sequence of points from $ X_\Sigma $ converging to $ x $ such that at least one of the two segments incident to the point has an angle with $ (ax) $ greater than $ 2c $. Let $L_1$ be such a point and let $\varepsilon$ be chosen not greater than $\rho$ for $c/5$ in Lemma~\ref{alm}. By item~\ref{__2} of Lemma~\ref{alm} both segments incident to $L_1$ has an angle with $(ax)$ greater than $9/5c$. Then the segment lying on the path from $L_1$ to $x$ should end with point $L_2\in X_\Sigma $. But the segment leaving $ L_2 $ must have an angle with $ (ax) $ greater than $ 8/5c $ and end with point $ L_3 \in X_\Sigma $. Thus for any arbitrarily small $ \bar \varepsilon$ there exists a broken line $ L_1-L_2-L_3-L_4-L_5 \subset \Sigma \cap B_{\bar \varepsilon} (x) $ such that all angles $ \angle ((L_jL_{j + 1}), (ax))> c $ for $j = 1,2,3,4 $. This is a contradiction with Proposition~\ref{brrrr}.
\end{proof}

\begin{lemma}
Let $x\in E_\Sigma$. 
\begin{enumerate}
\item Let $\Sigma'$ be a closure of a connected component of $\Sigma \setminus \{x\}$, let $(ax]$ be one-orientated energetic ray, let $\Sigma'$ contain an infinite sequence of energetic points of $(ax]$ converging to $x$, and doesn't contain (in the neighbourhood of $x$) any energetic points of other energetic rays.
Then there exists such 
$\rho_1>0$ that $\Sigma'\cap B_{\rho_1}(x)$ is an arc and contained in a cusp $Q \in \mathbb{Q}_{\rho_1}^{a,x}$.
\label{0}
\item If there exists such a cusp $Q\in \mathbb{Q}_\rho^{a,x}$ that $\Sigma \cap \partial Q \subset \{x\} \cup \partial B_\rho(x)$, then 
there exists such $\rho_2>0$, that $Q \cap B_{\rho_2}(x) \cap \Sigma$ is an arc or coincides with the point $\{x\}$.
\label{A}
\item  Let $\Sigma'$ be a closure of a connected component of $\Sigma \setminus \{x\}$, let $(ax]$ be an energetic ray, let $\Sigma'$ contain an infinite sequence of energetic points of $(ax]$ converging to $x$, and doesn't contain (in the neighbourhood of $x$) any energetic points of other energetic rays. Then there exists such 
$\rho_1>0$ that $\Sigma'\cap B_{\rho_1}(x)$ is an arc and is contained in a cusp $Q \in \mathbb{Q}_{\rho_1}^{a,x}$.
\label{AB}
\end{enumerate}
\label{uggi}
\end{lemma}
\begin{proof}
Let us prove item~\ref{0}. \begin{proof}
Without loss of generality we assume that $(ax] $ is the negative half of the abscissa axis, $ \rho $ is such that all points corresponding to the energetic points of $ \Sigma' \cap B_\rho(x) $ are located in the upper half-plane (we work in a neighbourhood of $x$, so it is possible due to item~\ref{1115} of Lemma~\ref{Losiu}) and $\sharp (\partial B_\rho(x) \cap \Sigma')=\oordball_x\Sigma'$ and $z\in \partial B_\rho(x) \cap \Sigma'$ is such that the path between $z$ and $x$ is contained in $\overline B_\rho(x)$.

We will monitor changes in the slope of the one-sided tangent to the arc $\breve{[zx[}$, where as a connected component we will choose a one containing $z$. By Corollary~\ref{1sdd} for sufficiently small $\varepsilon>0$ such one-sided tangent exists at each point of $\breve{[zx[}$ (and, due to the finite length of $\Sigma$, almost everywhere is contained in the tangent line at the same point).

Thus we can define 
$ \turn (\breve {[zx]}) $ as the upper limit (supremum) over all sequences of points of the curve:
  \[
 \turn(\breve{[zx]})=\sup_{n\in \N, z\preceq t_1 \prec \dots \prec t_n \prec x} \sum_{i=2}^{n} \widehat{t^i,t^{i-1}},
 \]
 where $ t ^ i $ denotes the ray of the one-sided tangent to the curve $ \breve {[zt_i]} \subset \breve {[zx[} $ at the point $t_i$, and $ {t_1, \dots, t_n}$ is the partition of the curve $ \breve{[zx[} $ in the order corresponding to the parameterization, for which $ z $ is the beginning of the curve and $ x $ is the end. In this case, the angle $ \widehat (t^i, t^{i+1}) \in [-\pi, \pi [$ between two rays is counted from ray $ t^i $ to ray $ t^{i+1} $; positive direction is counterclockwise.
 
Note that in view of $B_r(y(x))\cap \breve{[zx]} = \emptyset$, it is true that
 \[
|\turn ([zx])|  < 2 \pi.
 \]
Let us show that there are finitely many branching points on the path from $ z $ to $ x $. To do this, consider an arbitrary branching point $ D_2 $ on the arc $ \breve {[zx]} $.
Three segments $ [D_1D_2] $, $ [D_2D_3] $ and $ [D_2C] $ with pairwise angles $ 2 \pi / 3 $ emerge from the point $ D_2 $, where, without loss of generality, the broken line $ D_1-D_2-D_3 $ is a subset of the path from $ z $ to $ x $. Note that the contribution to the turn $ \breve {[zx]} $ of the point $ D_2 $ is $ - \pi / 3 $, since the counterclockwise angle $ \widehat{ D_1D_2D_3} $ is $ 2 \pi / 3 $. This is so, since the path that started from the segment $ [D_2C] $ must end with an energy point, and therefore must go to the upper half of the ball $ B_\rho (x) $.
 The same is true for an arbitrary branching point inside the path from $ z $ to $ x $.
 
For sufficiently small $\rho$ from Lemma~\ref{alm} and Corollary~\ref{sm_an} 
we get that the angle of the tangent direction at the arbitrary energetic point $\bar x$ on the path from $ z $ to $ x $ changes clockwise (the change at $\bar x$ does not exceed $ \pi / 6 + o_\rho(1)$), that is, the contribution to $ \turn $ is non-positive.

So there is a positive $ \rho_4 <\rho $ such that $ \breve {[zx]} \cap \overline B_{\rho_2} (x) $ does not contain branching points and thus is an arc.
Moreover, due to Lemma~\ref{dug_ok} there exists such positive $\rho_1<\rho_4$ and such a cusp $Q\in \mathbb{Q}^{a,x}_{\rho_1}$ that $\breve{[zx]}\cap B_{\rho_1}(x) \subset Q$ and $\sharp(\partial B_{\rho_1}(x) \cap \breve{[zx]})=\oordball_z\breve{[zx]}=1$.
\end{proof}
\begin{remark}
In this way, we have proved that a point $x \in E_\Sigma$ can be the limit of branching points only if it has a two-orientated energetic ray.
\label{remA}
\end{remark}

Let us prove item~\ref{A}.
\begin{proof}
Let $ \rho_2>0 $ satisfy the assertion of the Lemma~\ref{alm} for $ c = \frac{\pi}{100} $ and the assertion of the Lemma~\ref{zapred} for $ c = \frac{1}{20} $.In addition, we choose $ \rho_2> 0 $ such that
\[
\sharp(\Sigma \cap Q \cap \partial B_{\rho_2}(x))=\oordball_x(\Sigma\cap Q)\leq \oordball_x(\Sigma)\leq 3. 
\]
Clearly, $\oordball_x(\Sigma \cap Q)=1$: otherwise by co-area inequality one could decrease length of $\Sigma$ by replacing $\Sigma \cap Q$ with $T^Q$, which is defined in the proof of Lemma~\ref{angles2} and has a length $\rho_2 +o(\rho_2)$. Then energy will not increase and connectivity will remain, so we have a contradiction with an optimality of $\Sigma$. 

 Thus we assume $\sharp (\partial B_{\rho_2}(x) \cap \Sigma) =1$.
Let $Q_1 \defeq Q \cap \overline{B_{\rho_2}(x)}$ and $z \defeq \partial Q_1 \cap \Sigma$. 
We have to check only the case when the energetic ray $(ax]$ is two-orientated (if $(ax]$ is not an energetic ray at all everything is clear).

Assume $Q \cap \Sigma$ contains infinite number of branching points, then it also contains infinitely many terminal points. 

As $\sharp (\partial B_{\rho_2}(x) \cap \Sigma)=1$, each branching point in the path from $z$ to $x$ is connected with a terminal point by a segment. Suppose the contrary. In this case, on the path from $ z $ to $ x $, there is a branching point $ A $ connected to a branching point $ B $ that does not lie on the path from $ z $ to $ x $. Moreover, at least one of the two segments $ \Sigma \setminus [AB] $ incident to $ B $ is not perpendicular or parallel to $ (ax] $ up to $ \pi / 100 $, so it inevitably ends with a branching point, from which at least one segment also originates that is not perpendicular or parallel to $ (ax] $ up to $ \pi / 100 $, and so on. Thus there is a path that starts with this segment and never ends, and the segments of the same parity are parallel. It is easy to see that the limit point $ w $ of this path must belong to $E_\Sigma$ and lie in the ball $ B_\rho (x) $, due to Remark~\ref{remA}, its energetic ray should be two-orientated. But then (for sufficiently small $\rho_2>0$) $\partial B_{|xw|}(x) \cap \Sigma \cap Q_1=\{w \}$ which contradicts to the existence of the path from $z$ to $x$ lying in $Q_1 \cap \overline{B_{\rho_2}(x)}$ and is not containing $w$. So we have a contradiction to the assumption of the existence on the path such a branching point, which is not connected to a terminal point.

Thus, by item~\ref{__1} of Lemma~\ref{alm} in every branching point at the path from $z$ to $x$ the angle between a one-sided tangent to a path and a line $(ax)$ is modulo belongs to the interval $[\pi/6-\pi/100, \pi/6+\pi/100]$. Let $D_1$ be a branching point at the path of $\Sigma$ from $z$ to $x$, and let $[D_1;D_2]$ be a maximal segment of the path of $\Sigma$ from $D_1$ to $x$. As $|\angle (D_1D_2), (ax) -\pi/6|<\pi/100$, the point $D_2$ can be an isolated energetic point with order $2$ or a branching point only. If $D_2$ is a branching point then, similarly, the end of the maximal segment $[D_2;D_3]$ on the path, can be an isolated energetic point with order $2$ or a branching point only, and the angle $\angle((D_2D_3),(ax))=\pi/3-\angle((D_1D_2),(ax))$, and so belongs to the interval $[\pi/6-\pi/100, \pi/6+\pi/100]$. If $D_2$ is an isolated energetic point with order $2$ then, by item~\ref{__2} of Lemma~\ref{alm}
\[
\angle((D_2D_3),(ax)) \in [\angle((D_1D_2),(ax))-\pi/100, \angle((D_1D_2),(ax))+\pi/100] \subset [\pi/6-\pi/50,\pi/6+\pi/50],
\]
where $[D_2;D_3]$ is the maximal segment on the path from $D_2$ to $x$.
Thus, the path from $ D_1 $ to $ x $ starts with a polyline. In this case, the angle between the one-sided tangent and the straight line $ (ax) $ at the nodes of the polyline either belongs to the interval $ [\pi / 6- \pi / 100, \pi / 6 + \pi / 100] $ (if the node was a branching point), or changes by an amount not exceeding $ \pi / 100 $. Thus, in view of the Proposition~\ref {brrrr}, among the first four nodes of the polyline after $D_1$, there must be a branching point $ C $. Moreover, since all one-sided tangents to the path points between $ D_1 $ and $ C $ have an angle with $ (ax) $ greater than $ \pi / 6- \pi / 20> \pi / 10 $, there will be no points from the set $E_\Sigma$.
Repeating the reasoning for the point $ C $ as $ D_1 $, and then iterating, we get the path $ \sigma $, all one-sided tangents along which have an angle with $ (ax) $ greater than $ 1/10 $. Let $ \bar x $ be the limit point of such a path ($ \bar x $ belongs to the path from $ z $ to $ x $ and may coincide with $ x $). It is clear that $ \oordball_{\bar x}\sigma= 1 $. Let $ \rho_0> 0 $ be sufficiently small and such that
\[
\sharp (\sigma \cap \partial B_{\rho_0}(\bar x))=1.
\]
Define \[z_0=z_0(\rho_0) \defeq \partial B_{\rho_0}(\bar x) \cap \sigma.\]
Let $(a_0\bar x]$ be an energetic ray  of $\bar x$ closest to an infinite sequence of energetic points of $\sigma \cap B_{\rho_0}(\bar x)$.
Then, as $(a_0\bar x]$ due to item~\ref{0} and Remark~\ref{star} is two-orientated energetic ray, then
\[
\angle z_0\bar xa_0=o_{\rho_0}(1)
\]
and 
\begin{equation}
   \H(\sigma \cap B_{\rho_0}(\bar x))=\rho_0+o_{\rho_0}(1).
    \label{smll}
\end{equation}
On the other hand $\angle ((ax], (a_0\bar x])=o_{\rho_2}(1)$  and thus, for sufficiently small $\rho_0$ and $\rho_2$, all one-sided tangents to $\sigma \cap B_{\rho_0}(\bar x)$ have angles with $(a_0\bar x]$, greater than $\pi/10$. But it means that 
\[
\H(\sigma \cap B_{\rho_0}(\bar x)) \geq \frac{Pr_{(a_0\bar x]}(\sigma \cap B_{\rho_0}(\bar x))}{\cos \frac{\pi}{10}}\geq \frac{Pr_{(a_0\bar x]}([z_0\bar x])}{\cos \frac{\pi}{10}} \geq \frac{\cos \angle  z_0\bar x a_0}{\cos \frac{\pi}{10}}|z_0 \bar x|=\frac{\cos (o_{\rho_0}(1))}{\cos \frac{\pi}{10}}\rho_0.
\]
So we get a contradiction with~\eqref{smll}. 
\end{proof}
And now we have to prove item~\ref{AB}. 
\begin{proof}
If $(ax]$ is one-orientated energetic ray, then item~\ref{0} implies the statement. If $(ax]$ is two-orientated then, by Remark~\ref{star}, $\Sigma' \cap B_\rho(x)$ is contained in a cusp for sufficiently small $\rho>0$ and then item~\ref{A} of Lemma~\ref{uggi} implies the result.  

\end{proof}

\end{proof}
\begin{corollary}
Item~\ref{__3} of Lemma~\ref{alm} can be formulated as follows: if $\bar {x} \in E_\Sigma \setminus x $, then $ \exists \varepsilon : B_\varepsilon (\bar x) \cap \Sigma $ is an arc contained in the union of two cusps with axes almost parallel to $(ix)$ where these axes are one-sided tangents to $\Sigma$ at the point $\bar x$.
\label{notv}
\end{corollary}
\begin{remark}
As in Corollary~\ref{sm_an} in the notations of Corollary~\ref{notv} a path in $\Sigma$ from $\bar x$ to $x$ has a one-sided tangents almost parallel (not almost complementary) to $(ix]$.
\label{retv}
\end{remark}

And now we are finally ready to prove the following lemma, in which it is stated that no connected component of the $ \ Sigma \setminus \{ x \}$ can contain infinite sequences of energy points of two different energetic rays converging to $ x $.
\begin{lemma}
Let $\Sigma$ be the a minimizer, $(ax]$ and $(bx]$ be two different energetic rays with vertex at $x$ and cusps $Q^{a,x}_\rho \in \mathbb{Q}^{a,x}_\rho$ and $Q^{b,x}_\rho \in \mathbb{Q} ^{b,x}_\rho$ be such that $G_\Sigma \cap \overline{B_\rho(x)} \subset  Q^{a,x}_\rho \cup Q^{b,x}_\rho$ for some $\rho>0$. Let $\Sigma'$ be the connected component of $\Sigma \setminus \{x\}$ containing an infinite sequence of points from the set $G_\Sigma \cap  Q^{a,x}_\rho$, converging to $x$. Then there exists such $\rho_3>0$, that 
\[
\Sigma' \cap G_\Sigma \cap  Q^{b,x}_\rho \cap \overline {B_{\rho_3}(x)} = \{x\}.
\]
\label{eva}
\end{lemma}
The sketch of the proof is as follows: suppose the opposite, let $ \Sigma '$ contain an infinite sequence of points from the set $G_\Sigma \cap  Q^{b,x}_\rho$, converging to $x$. Then
\begin{enumerate}
\item[Step 1.] There exists such $\rho_1>0$, that $\Sigma' \cap \overline {B_{\rho_1}(x)} \setminus \{ x\}$ doesn't contain the points from $X_\Sigma$ with order $2$ except the ones with incident segments $o_\rho(1)$-almost parallel to the closest energetic ray $(ix]$.
\item[Step 2.] There exists such $\rho_2>0$, that $\Sigma' \cap \overline {B_{\rho_2}(x)}\setminus \{ x\}$ doesn't contain any energetic points with order $2$
\item[Step 3.] The angle $\angle ((ax],(bx])$ can be only $\pi/3$, $2 \pi/3$ or $\pi$.
\item[Step 4.] Obtaining a contradiction in each of the three possible cases.
\end{enumerate}
\begin{proof}
In view of the assumption that the assertion of the lemma is incorrect, we can assume that each of the energetic rays is one-orientated. Let's denote the value of the internal angle between the rays by $ \alpha$ (id est the angle doesn't contain any points corresponding to $x$): $ \alpha \defeq \angle ((ax],(bx]) \leq \pi $. It is possible, without loss of generality, to consider $ \rho_1 \leq \rho $ such that the statements of Lemma~\ref{alm} for $ c \defeq \alpha / 100 $ hold and that for every point $z \in Q^{i,x}_{\rho} \cap \overline{B_{\rho_1}(x)}$ the inequality $|\angle zxi |\leq \alpha/100$ holds, where $i=a,b$. We also get such a small $\rho_1$ that all points corresponding to the same energetic ray (excluding $x$) are located at the same semi-plane with respect to the line containing this energetic ray (It is possible due to item~\ref{1115} of Lemma~\ref{Losiu}). 
In the proof we still use the formulation ``the values \ldots are equal to within $ \bar c $ '', which mean that the absolute value of the difference between the values does not exceed $ \bar c $. 
Let area $U \ni x$ be such that $U \subset B_{\rho_1}(x)$ and $\sharp (\partial U \cap \Sigma)=\oord_x\Sigma$. Then $\sharp (\partial U \cap \Sigma')=1$, and hence for each point $\bar x \in \Sigma' \cap U \setminus{x}$ there exists only one path to $\partial U$, and the path from $\bar x$ to $x$ is contained in $U \subset B_\rho(x)$. 

\begin{itemize}

\item{Step 1.} For any point $ \bar {x} \in G_\Sigma \cap \Sigma'\cap U \setminus \{x \} $ with $\oord_x\Sigma =2$, the tangent rays to $ \bar x $ are $\alpha/10$-almost parallel to the nearest energetic ray $(ix]$.


Assume the contrary: let there exists such an energetic point 
$\bar x \in B_{\rho_1}(x) \setminus \{x\}$ with order $2$ that it has a tangent ray with angle with the closest to $\bar x$ energetic ray greater then $\alpha/10$. Thus by item~\ref{AB} Lemma~\ref{alm} the point $\bar x$ belongs to $X_\Sigma$ and hence there exists nontrivial segments $[s\bar x], [t\bar x] \subset \Sigma$. Without loss of generality let $[\bar x t]$ be the maximal segment in the path from $\bar x$ to $x$. Thus, by Corollary~\ref{sm_an}, as $\bar x \in Q^{i,x}_{\rho_1}$, $\angle ((t\bar x],(i x]) \leq \pi/6 +\alpha/100$. So
\begin{equation}
\alpha/10 <\angle ((t\bar x],( x \bar x])\leq \pi/6 +\alpha/50.
\label{hu}
\end{equation}
Without loss of generality let $i \defeq a$. The point $t$ can not be an energetic point of the energetic ray $(ax]$ due to Lemma~\ref{alm} and Corollaries~\ref{sm_an} and~\ref{notv}. It also can not coincide with $x$ due to definition of $Q^{a,x}_{\rho_1}\ni x$. Thus point $t$ can be a branching point or an energetic point (of order $2$) of the energetic ray $(bx]$.

 
 \begin{figure}
\begin{center}

\definecolor{uuuuuu}{rgb}{0.26666666666666666,0.26666666666666666,0.26666666666666666}
\definecolor{xdxdff}{rgb}{0.49019607843137253,0.49019607843137253,1.}
\definecolor{qqzzqq}{rgb}{0.,0.6,0.}
\definecolor{qqqqff}{rgb}{0.,0.,1.}
\begin{tikzpicture}[line cap=round,line join=round,>=triangle 45,x=1.0cm,y=1.0cm]
\clip(-5.76,-4.4) rectangle (5.6,4.36);
\draw [shift={(3.08,0.32)},color=qqzzqq,fill=qqzzqq,fill opacity=0.1] (0,0) -- (160.9065079995144:0.6) arc (160.9065079995144:203.87737522756086:0.6) -- cycle;
\draw [shift={(-2.361086323957324,2.2034529582929205)},color=qqzzqq,fill=qqzzqq,fill opacity=0.1] (0,0) -- (-47.57914750028342:0.6) arc (-47.57914750028342:-19.09349200048562:0.6) -- cycle;
\draw (3.08,0.32)-- (-4.72,3.02);
\draw (3.08,0.32)-- (-4.42,-3.);
\draw [->] (-4.62,2.54)-- (-2.76,2.64);
\draw [->] (-2.76,2.64)-- (-1.7,1.48);
\draw (-1.7,1.48)-- (0.42,1.64);
\draw (-1.7,1.48)-- (-2.94,-2.3);
\draw (-2.94,-2.3)-- (-2.52,-3.28);
\draw [->] (-2.94,-2.3)-- (-3.6,-2.3);
\draw [->] (-1.7,1.48) -- (-1.984168738471409,0.6137436843371562);
\draw [->] (-2.7047092790119494,-1.5827428021493297) -- (-2.94,-2.3);
\draw [dash pattern=on 3pt off 3pt] [->] (-3.6,-2.3)-- (-4.08,-1.56);
\draw [dash pattern=on 3pt off 3pt] (-3.6,-2.3)-- (-3.76,-2.58);
\begin{scriptsize}
\draw [fill=qqqqff] (3.08,0.32) circle (2.5pt);
\draw[color=qqqqff] (3.22,0.68) node {$x$};

\draw [fill=qqqqff] (-4.62,2.54) circle (2.5pt);
\draw[color=qqqqff] (-4.57,2.74) node {$s$};
\draw[color=black]  (1.66,1.04) node {$a$};
\draw[color=black] (-0.74,-0.98) node {$b$};
\draw[color=qqzzqq] (2.3, 0.3) node {$\alpha$};
\draw [fill=qqqqff] (-2.76,2.64) circle (2.5pt);
\draw[color=qqqqff] (-2.56,3.) node {$x^a_k$};
\draw [fill=qqqqff] (-1.7,1.48) circle (2.5pt);
\draw[color=qqqqff] (-1.56,1.80) node {$t$};
\draw [fill=qqqqff] (0.42,1.64) circle (2.5pt);
\draw[color=qqqqff] (0.74,2.) node {$x^a_{k+1}$};
\draw [fill=qqqqff] (-2.94,-2.3) circle (2.5pt);
\draw[color=qqqqff] (-2.65,-1.94) node {$D$};
\draw [fill=qqqqff] (-2.52,-3.28) circle (2.5pt);
\draw [fill=qqqqff] (-3.6,-2.3) circle (2.5pt);
\draw[color=qqzzqq] (-1.42,2.34) node {$\delta = \pi/6$};
\end{scriptsize}
\end{tikzpicture}

\caption{Figure for Step 1 of the proof of the lemma~\ref{eva}}
\label{norm1}
\end{center}
\end{figure}
\begin{figure}
\begin{center}

\definecolor{qqwuqq}{rgb}{0.,0.39215686274509803,0.}
\definecolor{uuuuuu}{rgb}{0.26666666666666666,0.26666666666666666,0.26666666666666666}
\definecolor{qqqqff}{rgb}{0.,0.,1.}
\begin{tikzpicture}[line cap=round,line join=round,>=triangle 45,x=1.0cm,y=1.0cm]
\clip(-4.3,-2.46) rectangle (7.06,6.3);
\draw [shift={(-1.5547267260185111,3.9511797653440244)},color=qqwuqq,fill=qqwuqq,fill opacity=0.1] (0,0) -- (164.56247582346808:0.6) arc (164.56247582346808:196.74197038529311:0.6) -- cycle;
\draw [shift={(-0.7851861245803811,3.7386706452816116)},color=qqwuqq,fill=qqwuqq,fill opacity=0.1] (0,0) -- (-52.74471054642292:0.6) arc (-52.74471054642292:-15.43752417653194:0.6) -- cycle;
\draw [shift={(1.8152832336344478,0.31953500762878106)},color=qqwuqq,fill=qqwuqq,fill opacity=0.1] (0,0) -- (127.25528945357708:0.6) arc (127.25528945357708:197.51387416437143:0.6) -- cycle;
\draw [shift={(2.2003747945205476,0.44105643835616437)},color=qqwuqq,fill=qqwuqq,fill opacity=0.1] (0,0) -- (17.513874164371412:0.6) arc (17.513874164371412:88.05851360908564:0.6) -- cycle;
\draw [shift={(6.38,1.76)},color=qqwuqq,fill=qqwuqq,fill opacity=0.1] (0,0) -- (164.56247582346805:0.6) arc (164.56247582346805:197.5138741643714:0.6) -- cycle;
\draw (-3.18,4.4)-- (6.38,1.76);
\draw (6.38,1.76)-- (-3.,-1.2);
\draw (-3.52,3.36)-- (-1.06,4.1);
\draw (2.18,-0.16)-- (-1.06,4.1);
\draw (2.18,-0.16)-- (2.24,1.46);
\draw [dash pattern=on 3pt off 3pt] (2.24,1.46)-- (2.261400286944046,2.2413084648493546);
\begin{scriptsize}
\draw [fill=qqqqff] (6.38,1.76) circle (2.5pt);
\draw[color=qqqqff] (6.52,2.12) node {$x$};
\draw[color=black] (1.56,2.94) node {$a$};
\draw[color=black]  (-2.64,-0.76) node {$b$};
\draw [fill=qqqqff] (-3.52,3.36) circle (2.5pt);
\draw[color=qqqqff] (-3.38,3.72) node {$s$};
\draw [fill=qqqqff] (-1.06,4.1) circle (2.5pt);
\draw[color=qqqqff] (-0.92,4.46) node {$x^a_k$};
\draw [fill=qqqqff] (2.18,-0.16) circle (2.5pt);
\draw[color=qqqqff] (2.47, -0.2) node {$x_b^l$};
\draw[color=qqwuqq] (-2.48,3.87) node {$\delta$};
\draw[color=qqwuqq] (0.38,3.12) node {$\delta + o_\rho(1)$};
\draw[color=qqwuqq] (0.2,0.5)  node {$\alpha + \delta + o_\rho(1)$};
\draw[color=qqwuqq] (3.24,1.2) node {$\alpha + \delta + o_\rho(1)$};
\draw[color=qqwuqq] (5.55,1.82) node {$\alpha$};
\end{scriptsize}
\end{tikzpicture}

\caption{Figure for Step 1 of the proof of the lemma~\ref{eva}}
\label{norm2}
\end{center}
\end{figure}
\begin{figure}[h]
\begin{center}

\definecolor{uuuuuu}{rgb}{0.26666666666666666,0.26666666666666666,0.26666666666666666}
\definecolor{xdxdff}{rgb}{0.49019607843137253,0.49019607843137253,1.}
\definecolor{qqzzqq}{rgb}{0.,0.6,0.}
\definecolor{qqqqff}{rgb}{0.,0.,1.}
\begin{tikzpicture}[line cap=round,line join=round,>=triangle 45,x=1.0cm,y=1.0cm]
\clip(-5.76,-4.5) rectangle (5.6,4.36);
\draw [shift={(3.08,0.32)},color=qqzzqq,fill=qqzzqq,fill opacity=0.1] (0,0) -- (160.9065079995144:0.6) arc (160.9065079995144:203.87737522756086:0.6) -- cycle;
\draw [shift={(-2.3887081339712917,2.213014354066986)},color=qqzzqq,fill=qqzzqq,fill opacity=0.1] (0,0) -- (-48.9909130984298:0.6) arc (-48.9909130984298:-19.093492000485618:0.6) -- cycle;
\draw (3.08,0.32)-- (-4.72,3.02);
\draw (3.08,0.32)-- (-4.42,-3.);
\draw [->] (-4.3,2.48)-- (-2.76,2.64);
\draw [->] (-2.76,2.64)-- (-1.56,1.26);
\draw (-1.56,1.26)-- (-0.78,1.5);
\draw (-1.56,1.26)-- (-2.84,-2.68);
\draw (-2.84,-2.68)-- (-2.52,-3.28);
\draw [->] (-2.84,-2.68)-- (-4.04,-2.4);
\draw [->] (-1.56,1.26) -- (-1.8533354719704869,0.3570767503408451);
\draw [->] (-2.5971192557542704,-1.9323827091186139) -- (-2.84,-2.68);
\draw (-0.78,1.5)-- (-0.54,2.08);
\draw [dash pattern=on 3pt off 3pt] (-0.78,1.5)-- (-0.36,1.08);
\draw [dash pattern=on 3pt off 3pt] [->] (-4.04,-2.4)-- (-4.4,-1.52);
\draw [dash pattern=on 3pt off 3pt] (-4.04,-2.4)-- (-4.36,-2.58);
\begin{scriptsize}
\draw [fill=qqqqff] (3.08,0.32) circle (2.5pt);
\draw[color=qqqqff] (3.22,0.68) node {$x$};
\draw[color=black] (0.66,1.04) node {$a$};
\draw[color=black] (-0.74,-1.12) node {$b$};
\draw[color=qqzzqq] (2.25, 0.3) node {$\alpha$};
\draw [fill=qqqqff] (-4.3,2.48) circle (2.5pt);
\draw[color=qqqqff] (-4.12,2.7) node {$s$};
\draw [fill=qqqqff] (-2.76,2.64) circle (2.5pt);
\draw[color=qqqqff] (-2.56,3.) node {$x^a_k$};
\draw [fill=qqqqff] (-1.56,1.26) circle (2.5pt);
\draw[color=qqqqff] (-1.42,1.62) node {$t$};
\draw [fill=qqqqff] (-0.78,1.5) circle (2.5pt);
\draw[color=qqqqff] (-0.86,1.26) node {$C$};
\draw [fill=qqqqff] (-2.84,-2.68) circle (2.5pt);
\draw[color=qqqqff] (-2.5,-2.32) node {$D$};
\draw [fill=qqqqff] (-2.52,-3.28) circle (2.5pt);
\draw [fill=qqqqff] (-4.04,-2.4) circle (2.5pt);
\draw[color=qqzzqq] (-1.44,2.3) node {$\delta = \pi/6$};
\draw [fill=qqqqff] (-0.54,2.08) circle (2.5pt);
\draw[color=qqqqff] (-0.4,2.44) node {$x^a_{k+1}$};
\end{scriptsize}
\end{tikzpicture}

\caption{Figure for Step 1 of the proof of the lemma~\ref{eva}}
\label{norm11}
\end{center}
\end{figure}

\begin{enumerate}
\item Let $t$ be a branching point and let $C$ be such a point that $[tC]$ is a maximal segment in $\Sigma$ 
and the angle $\angle ((tC], (ax])=\pi/3-\angle ([t\bar x),(x \bar x])$. Thus $C$ can not be an energetic point of the energetic ray $(ax]$ of order $1$. Then the following cases are possible:
\begin{enumerate}
    \item $C$ is an energetic point of the energetic ray $(ax]$ of order $2$. By Corollary~\ref{sm_an} $ \angle ((tC], (ax]) \leq \pi/6+\alpha/100$. 
    Hence $\angle ((tC], (ax])$ is equal to $\pi/6$ within $\alpha/50$.
    \item $C$ is an energetic point of the energetic ray $(bx]$
    \begin{enumerate}
        \item $C$ is an energetic point of the energetic ray $(bx]$ of order $1$.
        
        If $ C $ were an energetic point of the ray $ (bx] $ of degree $ 1 $, then there would be a contradiction with the fact that $ [\bar x t] $ lies on the path from $ \bar x $ to $ x $: polyline $ [ \bar x t] \cup [tC] $ and four empty balls of radius $ r $ (empty balls touching the points $ \bar x $ and $ C $ with centers at the corresponding points, as well as two empty balls, tangent to the energetic rays $ (ax] $ and $ (bx] $ centered at the limit points of the points corresponding to the energetic points of the rays) entail the absence of a path from $ \bar x $ to $ x $ in $ \Sigma'$. So we get a contradiction.
        
        \item $C$ is an energetic point of the energetic ray $(bx]$ of order $2$.
        
        Let the point $ C $ be an energetic point of the ray $ b $ of degree $ 2 $, then $ \angle ((tC], (bx]) \leq \pi / 6 + \alpha / 100 $. Let $ [tD] $ be the third maximal segment incident to the branching point $ t $. Note that 
         $ \angle ([Dt), (bx])>2\pi/3$. Consider two paths inside $ U $: the path starting from $ t-\bar x $ and the path starting from $ t-D $: we will show that both of them have to contain the single point from $\partial U \cap \Sigma'$, which is impossible: 
        Let the first path does not contain any energetic points of the ray $ (bx] $, so the direction of the tangent along the path changes only at the branching points and energetic points of the ray $ (ax] $ of degree $ 2 $. We require that at the branching points the path continues not by the segment that is $\alpha/50$-almost perpendicular to $ (ax] $.  Then this path does not contain any energetic points of degree $ 1 $ of the ray $ (ax]$ and all the angles between $ (ax) $ and the one-sided tangent to $ \Sigma $ along the path do not exceed $ \pi / 6 + \alpha / 100 $. Clearly this path can not end by a limit point, because by Remark otherwise $\bar x$ has two-orientated energetic ray, which is impossible for sufficiently small $\rho_1>0$ as $(ax]$ and $(bx]$ are one-orientated (in what follows we will omit this reasoning). Thus, this path cannot end with an energetic point of degree $ 1 $ and  point $ x $, but then the path inevitably passes through a single point $ \Sigma'\cap \partial U $. Similar reasoning is true for some path starting with $ t-D $ under the assumption that it does not contain energetic points $ (ax] $.
        
 Thus, it is impossible that at the same time the path starting from $ t-\bar x$ does not contain energetic points of the energetic ray $ (bx] $, and the path starting from $ t-D $ does not contain energetic points of the ray $ (ax] $. But if the path, starting with $ t - \bar x $, contains an energetic point $ (bx] $, then the second path is bounded by the first path and empty balls of radius $ r $, that is, it does not contain the energetic points of the ray $ (ax] $ and of the boundary $\partial U$, which is impossible. Similarly, the situation is impossible when the path starting with $ t-D $ contains an energetic point of the ray $ (ax] $.

        \end{enumerate}
    \item  $C$   is a branching point 

 In this case, the path starting from the segments $ [\bar x t] $ and $ [tC] $ and continuing with turns in the same direction can end only if $ C $ is connected by a segment with the energetic point of the energetic ray $ (ax] $ of order $ 1 $ , we denote it by $ x^a_2 $ (this point can not be of order $2$ as $\angle ((Cx^a_2],(ax]) \geq 2\pi/3 - \angle([t\bar x)(ax]) \geq \pi/2-\alpha/50>\pi/6+\alpha/100$. Then the angle $ \angle ((Cx^a_2), (ax)) $ is equal to $ \pi / 2 $ up to $ \alpha / 100 $. In this case, considering the quadrilateral bounded by the lines $ (ax) $, $ (\bar x t), (tC), (Cx^a_2) $, we get 
$\angle  ([t\bar x)(ax])=2\pi/3-\angle ((Cx^a_2),(ax))$,
which, up to $ \alpha / 100 $, is equal to $ \pi / 6 $ (see Fig.~\ref{norm11}).
            \end{enumerate}
Thus, in any case $ \angle ((\bar xt], (ax]) $ is equal to $ \pi / 6 $ up to $ \alpha / 50 $. Let $ [tD] $ be the remaining maximum segment with the end at the point $ t $. Then $ \angle ((tD), (ax)) $ is equal to $ \pi / 2 $ within $ \alpha / 50 $, then either $ [tD) $ and $ (bx ] $ do not intersect (if $ \alpha + \angle ((tD), (ax))> \pi $), or $ \angle ((tD], (bx]) $ equals $ \pi / 2- \alpha $ up to $ \alpha / 50 $. In this case, the point $ D $ cannot be an energetic point of the ray $ (bx] $ of order $1$. The point $D$ also can not be a energetic point of order $2$ or a branching point as in these cases there are two paths in $ \Sigma'$ (a path starting with $ t-\bar x $ and a path starting with $ t-D $) that cannot coexist (the proof is similar to that discussed above).

This contradicts the assumption that $ t $ is a branching point.
\item Let $t$ be an energetic point of the energetic ray $(bx]$.
Note that
\[
\angle((\bar x t],(bx])= \angle (( \bar x t],(b x]) + \alpha \leq \pi/6+\alpha/100,
\]
in view of Lemma~\ref{alm} and Corollary~\ref{sm_an}, since $t$ can only be an energetic point of the ray $ (bx] $ of degree $ 2 $ belonging to the set $ X_\Sigma $. The previous item entails that the maximum segment on the path to $\Sigma $ from $t$ to $ x $ ends with an energetic point of the energetic ray $ (ax] $, we denote it by $ x^a_2 $. This point is of order $2$ and $\angle((tx^a_2),(ax))>\alpha/50$ thus the previous item entails the existence in $\Sigma'$ a poly-line $\bar x -t-x^a_2-x^b_2$, where $x^b_2$ is an energetic point of order $2$ of the energetic ray $(bx]$ such that the segment $[x^a_2x^b_2]$ is a maximal segment in a path from $x^a_2$ to $x$ in $\Sigma$. But then the previous item entails the existence in $\Sigma'$ a poly-line $\bar x -t-x^a_2-x^b_2-x^a_3$, where $x^a_3$ is an energetic point of order $2$ of the energetic ray $(ax]$ such that the segment $[x^b_2x^a_3]$ is a maximal segment in a path from $x^a_2$ to $x$ and so on, so there exists an infinite poly-line with alternating energetic points of two energetic rays as the vertices. Wherein at every step the angle between a segment of poly-line and the energetic ray increases by  $\alpha$ within $\alpha/100$ (see Fig.~\ref{norm2}):
$\angle((x^a_kx_k^b],(bx])= \angle ((x_k^bx_k^a], (xx_k^a]) + \alpha $ within $\alpha/100$ and $\angle((x^b_kx^a_{k+1}],(ax])= \angle ((xx^b_k],(x_{k+1}^ax^b_k])+\alpha \geq \angle ((x^a_kx^b_k],(bx])+ \alpha -\alpha/100= \angle ((x_k^bx_k^a], (xx_k^a]) + 2\alpha-\alpha/100$ within $\alpha/100$
for every $k\in \N$ if $x^a_1 \defeq x$ and $x^b_1 \defeq t$.
Thus, there exists a number $ k $ such that $ \angle ((x ^ i_kx ^ j_k), (jx))> \pi / 6 + \alpha / 50 $, where $ \{i, j \} = \{a, b \} $, but then the point $ x^k_j $ cannot be an energetic point of degree $ 2 $ of the ray $ (jx] $. The resulting contradiction completes the proof of the step.
\end{enumerate}


\item[Step 2.] Every energetic point of $\Sigma' \cap U \setminus \{x\}$ belongs to the set $X_\Sigma$ and has the order $1$.
\begin{proof}
Assume the contrary: due to item~\ref{__1} of Lemma~\ref{alm} and step 1 it means that there exists such $\bar x \in \Sigma' \cap U \setminus \{x\}$ of order $2$ that it belongs to $E_\Sigma$ (and hence by Lemma~\ref{alm} and Corollary~\ref{notv} there exist one-sided tangents at $\bar x$ $\alpha/100$-almost parallel to $(ix]$, the closest energetic ray to $\bar x$) or to $X_\Sigma$ with both incident segments $\alpha/10$-almost parallel to $(ix]$. 

Consider the path from $\bar x$ to $x$. This path is contained in $U$ by definition of $U$ and its points by Remark~\ref{remA} can not be a limit point of the branching points. Thus if this path contains some branching points one can pick the first branching point $C_0$. Then, in view of compactness of $E_\Sigma$, we can take the last energetic point $B$ from the path from $\bar x$ to $C_0$  (it may coincide with $\bar x$). Thus the path from $B$ to $C_0$ has to be a segment and, due to Lemma~\ref{alm}, Corollary~\ref{notv} and step 1, this segment is contained in a line, $\alpha/100$-almost parallel to $(jx]$, energetic ray closest to $B$ for $j\in \{a,b\}$. Without loss of generality let $j=a$. Consider the maximal segment $[C_0C_1]\subset \Sigma'$, such that $\angle((ax](C_0C_1])$ is equal to $\frac{\pi}{3}$ within $\alpha/100$. Hence, by Lemma~\ref{alm} and Corollary~\ref{sm_an} $C_1$ can not be an energetic point and thus it has to be a branching point. Then let $[C_1C_2] \subset \Sigma'$ be such a maximal segment that it is not contained in a almost parallel to $(ax)$ line and hence $\angle((ax](C_1C_2])$ is equal to $\frac{2\pi}{3}$ within $\alpha/100$. Thus $C_2$ has to be a branching point. But then the path, beginning by $C_0-C_1-C_2-\ldots$ and continuing by the alternating turns (id est the injective poly-line $\bigcup_{j=0} [C_j C_{j+1}]$, with $\angle((ax](C_{j}C_{j+1}])=\pi-\angle((ax](C_{j-1}C_{j}])$) can not ends.

If the path from $ \overline{x} $ to $ x $ does not contain any branching points, then the arc $ \breve{]\overline{x} x[} $ will also have no energetic points of the energetic ray $ (bx] $: otherwise, let $ x_b $ be the first such point; and let $ x_a $ be the last energetic point of the ray $ (ax] $ from this arc before $x_b$ ($x_a \in Q^{a,x}_{\rho_1}$ and $x_b \in Q^{b,x}_{\rho_1}$ are the different points because $Q^{a,x}_{\rho_1}\cap Q^{b,x}_{\rho_1} =\{x\}$). In this case, the arc $ \breve{x_a x_b} $ is a segment, contained in a line almost parallel to both energetic-rays, which is impossible:
\[
\angle ((x_a x_b), (ax)) \leq \alpha/10
\]
and
\[
\angle ((x_a x_b), (bx)) \leq \alpha/10
\]
entails
\[
\alpha=\angle((ax],(bx]) =\angle((ax),(bx))\leq \alpha/5,
\]
which is impossible for $\alpha>0$. 

Thus the path from $\bar x$ to $x$ doesn't contain any energetic points of the energetic ray $(bx]$ or any branching points, id est in each point of the path the tangent ray of $\Sigma'$ is $\alpha/10$ almost parallel to $(ax]$. Then there exists such open area $U_0 \subset B_{|\bar x x|}(x)$, that $x \in U_0$, $\sharp (\partial U \cap \Sigma ) = \oord_x \Sigma$. Then $U_0 \cap \Sigma' \subset \breve{\overline{x}x}$ doesn't contain any energetic points of the ray $(bx]$. So we get a contradiction with a definition of $\Sigma'$ for $\rho_2 \defeq \rho_1$.

\end{proof}

\item [Step 3]
Let $ U_1 \ni x $ be an open area of sufficiently small diameter (so that $ \Sigma'\cap (U \setminus U_1) $ contains sufficiently many energetic points of both energetic rays: for example at least three energetic points of each energetic ray) such that $ U_1 \subset U $, $ \sharp (\partial U_1 \cap \Sigma) = \oord_x \Sigma $.
In view of the previous steps, $\overline{\Sigma'\cap (U \setminus U_1)} $ is a locally minimal tree (with a finite number of vertices) without vertices of degree $ 2 $. Then there exist three lines with pairwise angles $ \pi / 3 $ such that each segment of $ \overline{\Sigma'\cap (U \setminus U_1)} $ is parallel to one of them.
Each endpoint of the Steiner tree $ \overline{\Sigma'\cap (U \setminus U_1)} $, except two points of $ (\partial U \cup \partial U_1) \cap \Sigma' $, lies on a segment belonging to the line $\alpha/100$-almost perpendicular to one of the rays. Therefore, since the set $ \Sigma'\cap (U \setminus U_1) $ contains the energetic points of both rays, the angle $ \alpha = \angle axb \leq \pi$ can only be $\frac{\pi}{3}$, $\frac{2\pi}{3}$ or $\pi$ within $\alpha/100$.

\item[ Step 4]
\begin{figure}[h]
    \centering
\definecolor{qqwuqq}{rgb}{0.,0.39215686274509803,0.}
\definecolor{xdxdff}{rgb}{0.49019607843137253,0.49019607843137253,1.}
\definecolor{qqqqff}{rgb}{0.,0.,1.}
\definecolor{ffqqqq}{rgb}{1.,0.,0.}
\definecolor{uuuuuu}{rgb}{0.26666666666666666,0.26666666666666666,0.26666666666666666}
\begin{tikzpicture}[line cap=round,line join=round,>=triangle 45,x=1.0cm,y=1.0cm]
\clip(-3.3,-2.7) rectangle (3.24,3.04);
\draw [->,color=ffqqqq] (0.,0.) -- (0.,-2.);
\draw [->,color=qqqqff] (0.,0.) -- (1.7320508075688772,-1.);
\draw [->,color=ffqqqq] (0.,0.) -- (1.7320508075688776,1.);
\draw [->,color=qqqqff] (0.,0.) -- (0.,2.);
\draw [->,color=ffqqqq] (0.,0.) -- (-1.7320508075688767,1.);
\draw [->,color=qqqqff] (0.,0.) -- (-1.732050807568878,-1.);
\draw [->,color=qqqqff] (0.,0.) -- (-1.7320508075688774,-1.);
\draw [dash pattern=on 5pt off 5pt,domain=-3.300000000000001:1.1407883832488646] plot(\x,{(--0.8675987569037055-0.*\x)/-1.31726896031426});
\draw [dash pattern=on 5pt off 5pt,domain=-3.300000000000001:1.1407883832488646] plot(\x,{(-0.8675987569037056--1.1407883832488648*\x)/-0.65863448015713});
\begin{scriptsize}
\draw [fill=uuuuuu] (0.,0.) circle (1.5pt);
\draw[color=ffqqqq] (0.12,-0.84) node {$1$};
\draw[color=qqqqff] (0.8,-0.26) node {$2$};
\draw[color=ffqqqq] (1.,0.74) node {$1$};
\draw[color=qqqqff] (0.16,1.0) node {$-1$};
\draw[color=ffqqqq] (-0.86,0.6) node {$-2$};
\draw[color=qqqqff] (-0.62,-0.6) node {$-1$};
\draw [fill=xdxdff] (1.1407883832488646,-0.6586344801571302) circle (2.5pt);
\draw[color=xdxdff] (1.38,-0.3) node {$x$};
\end{scriptsize}
\end{tikzpicture}
    \caption{Rose of wind for $\alpha = \pi/3$}
    \label{fig:roza1}
\end{figure}
\begin{figure}[h]
    \centering
\definecolor{qqwuqq}{rgb}{0.,0.39215686274509803,0.}
\definecolor{xdxdff}{rgb}{0.49019607843137253,0.49019607843137253,1.}
\definecolor{qqqqff}{rgb}{0.,0.,1.}
\definecolor{ffqqqq}{rgb}{1.,0.,0.}
\definecolor{uuuuuu}{rgb}{0.26666666666666666,0.26666666666666666,0.26666666666666666}
\begin{tikzpicture}[line cap=round,line join=round,>=triangle 45,x=1.0cm,y=1.0cm]
\draw [->,color=ffqqqq] (0.,0.) -- (0.,-2.);
\draw [->,color=qqqqff] (0.,0.) -- (1.7320508075688772,-1.);
\draw [->,color=ffqqqq] (0.,0.) -- (1.7320508075688776,1.);
\draw [->,color=qqqqff] (0.,0.) -- (0.,2.);
\draw [->,color=ffqqqq] (0.,0.) -- (-1.7320508075688767,1.);
\draw [->,color=qqqqff] (0.,0.) -- (-1.732050807568878,-1.);
\draw [dash pattern=on 5pt off 5pt] (0.9947818472546173,0.)-- (-0.28886206679760795,2.223336477965032);
\draw [dash pattern=on 5pt off 5pt] (0.9947818472546173,0.)-- (-0.2888620667976096,-2.2233364779650318);
\begin{scriptsize}
\draw [fill=uuuuuu] (0.,0.) circle (1.5pt);
\draw[color=ffqqqq] (0.12,-0.84) node {$1$};
\draw[color=qqqqff] (1.02,-0.26) node {$2$};
\draw[color=ffqqqq] (1.,0.74) node {$1$};
\draw[color=qqqqff] (0.16,1.16) node {$-1$};
\draw[color=ffqqqq] (-0.86,0.6) node {$-2$};
\draw[color=qqqqff] (-0.72,-0.4) node {$-1$};
\draw [fill=xdxdff] (0.9947818472546173,0.) circle (2.5pt);
\draw [fill=xdxdff] (0.9947818472546173,0.) circle (2.5pt);
\draw[color=xdxdff] (1.3,0.28) node {$x$};
\end{scriptsize}
\end{tikzpicture}
    \caption{Rose of wind for $\alpha = 2\pi/3$}
    \label{fig:roza2}
\end{figure}
\begin{figure}
    \centering
\definecolor{xdxdff}{rgb}{0.49019607843137253,0.49019607843137253,1.}
\definecolor{qqqqff}{rgb}{0.,0.,1.}
\definecolor{uuuuuu}{rgb}{0.26666666666666666,0.26666666666666666,0.26666666666666666}
\begin{tikzpicture}[line cap=round,line join=round,>=triangle 45,x=1.0cm,y=1.0cm]
\clip(-3.3,-2.8) rectangle (3.24,3.04);
\draw [->,line width=0.8pt,color=qqqqff] (0.,0.) -- (1.7320508075688772,-1.);
\draw [->,line width=0.8pt,color=qqqqff] (0.,0.) -- (0.,2.);
\draw [->,line width=0.8pt,color=qqqqff] (0.,0.) -- (-1.732050807568878,-1.);
\draw [line width=0.8pt,dash pattern=on 5pt off 5pt,domain=-3.300000000000001:1.1407883832488646] plot(\x,{(--1.7351975138074112-1.1407883832488643*\x)/-0.6586344801571301});
\draw [line width=0.8pt,dash pattern=on 5pt off 5pt,domain=1.1407883832488646:3.240000000000001] plot(\x,{(-1.7351975138074112--1.1407883832488646*\x)/0.65863448015713});
\begin{scriptsize}
\draw [fill=uuuuuu] (0.,0.) circle (1.5pt);
\draw[color=qqqqff] (0.97,-0.23) node {$1$};
\draw[color=qqqqff] (0.16,1.17) node {$-1$};
\draw[color=qqqqff] (-0.7,-0.2) node {$0$};
\draw [fill=xdxdff] (1.1407883832488646,-0.6586344801571301) circle (2.5pt);
\draw[color=xdxdff] (1.4,-0.66) node {$x$};
\end{scriptsize}
\end{tikzpicture}
 \caption{Rose of wind for $\alpha = \pi$}
    \label{fig:roza3}
\end{figure}
Let us construct the wind rose with the center inside the internal angle $\angle axb$ and the lines, containing perpendiculars to both energetic rays $(ax]$ and $(bx]$. We place the weights depending on the angle $\angle axb$:
\begin{itemize}
\item If $\alpha = \frac{\pi}{3}$ then place $1$ at both rays directed towards to $(ax]$ and $(bx]$ and perpendicular to it (see Fig.~\ref{fig:roza1}). Thus there are weights $-1$ at the complementary rays, and weights $2$ and $-2$ at the last line.
 \item If $\alpha = \frac{2\pi}{3}$ then place $1$ at the ray, perpendicular to $(ax]$ and directed towards to it, and place $2$ at the ray, perpendicular to $(bx]$ and directed towards to it (see Fig.~\ref{fig:roza2}). Thus there are weights $-1$ and $-2$ at the complementary rays, and weights $1$ and $-1$ at the last line.
 \item If $\alpha = \pi$ we place $1$ at the ray directed towards to $(ab)$  and perpendicular to it, and weights $-1$ and $0$ at the remaining ones (see Fig.~\ref{fig:roza3}). Thus there are weights $-1$, $1$ and $0$ at the complementary rays.
 \end{itemize}
 
 Now, as in Remark~\ref{main_rose_property}, we assign to each terminal vertex of the tree  $\overline{\Sigma \cap U}$ the weight of the ray of the corresponding wind rose, co-directed with the segment entering this vertex.

 
 
By Remark~\ref{main_rose_property}, sum of all weights of $\Sigma' \cap (U \setminus U_1)$ is equal to $0$. On the other hand, sum is positive, as there are arbitrarily many positive terms, and at most two negative. So we get a contradiction.

\end{itemize}

\end{proof}






\paragraph{Acknowledgments.}
Statements from Lemma~\ref{angles2} to Lemma~\ref{Losiu} are supported by ``Native towns'', a social investment program of PJSC ``Gazprom Neft''. Other statements of this work are supported by grant 16-11-10039 of the Russian Science Foundation. I would like to thank~Evgeny Stepanov for setting the problem and providing scientific guidance in 2010-2018, Danila Cherkashin, without whom this work would never have been completed, and Azat Miftakhov for an inspiring example of doing mathematics in any conditions.

\bibliographystyle{plain}
\bibliography{main}

\begin{thebibliography}{10}

\bibitem{BS1}
G.~Buttazzo and E.~Stepanov.
\newblock Optimal transportation networks as free {D}irichlet regions for the
  {M}onge-{K}antorovich problem.
\newblock {\em Ann. Sc. Norm. Super. Pisa Cl. Sci. (5)}, 2(4):631--678, 2003.

\bibitem{gilbert1968steiner}
E.~N. Gilbert and H.~O. Pollak.
\newblock Steiner minimal trees.
\newblock {\em SIAM Journal on Applied Mathematics}, 16(1):1--29, 1968.

\bibitem{hwang1992steiner}
Frank~K. Hwang, Dana~S. Richards, and Pawel Winter.
\newblock {\em The {S}teiner tree problem}, volume~53.
\newblock Elsevier, 1992.

\bibitem{kuratowski2014topology}
K.~Kuratowski.
\newblock {\em Topology}, volume~1.
\newblock Elsevier, 2014.

\bibitem{lemenant2011regularity}
Antoine Lemenant.
\newblock About the regularity of average distance minimizers in
  $\mathbb{R}^2$.
\newblock {\em J. Convex Anal}, 18(4):949--981, 2011.

\bibitem{L}
Antoine Lemenant.
\newblock A presentation of the average distance minimizing problem.
\newblock {\em Journal of Mathematical Sciences}, 181(6):820--836, 2012.

\bibitem{mir}
M.~Miranda, Jr., E.~Paolini, and E.~Stepanov.
\newblock On one-dimensional continua uniformly approximating planar sets.
\newblock {\em Calc. Var. Partial Differential Equations}, 27(3):287--309,
  2006.

\bibitem{PaoSte04max}
E.~Paolini and E.~Stepanov.
\newblock Qualitative properties of maximum distance minimizers and average
  distance minimizers in {${\Bbb R}^n$}.
\newblock {\em J. Math. Sci. (N. Y.)}, 122(3):3290--3309, 2004.
\newblock Problems in mathematical analysis.

\bibitem{PaoSte13Steiner}
E.~Paolini and E.~Stepanov.
\newblock Existence and regularity results for the {S}teiner problem.
\newblock {\em Calc. Var. Partial Differential Equations}, 46(3-4):837--860,
  2013.

\bibitem{paolini2015example}
E.~Paolini, E.~Stepanov, and Y.~Teplitskaya.
\newblock An example of an infinite {S}teiner tree connecting an uncountable
  set.
\newblock {\em Advances in Calculus of Variations}, 8(3):267--290, 2015.

\end{thebibliography}

\end{document}